\documentclass[AMA,STIX1COL]{WileyNJD-v2}

\articletype{Article Type}%

\received{26 April 2016}
\revised{6 June 2016}
\accepted{6 June 2016}

\begin{document}

\title{\;Isoperimetric Problem and Weierstrass Necessary Condition for Fractional Calculus of Variations}

\author[1]{Shakir Sh. Yusubov}

\author[2]{Shikhi Sh. Yusubov}

\author[3,4]{Elimhan N. Mahmudov*}


\address[1]{\orgdiv{Department of Mechanics and Mathematics}, \orgname{Baku State University}, \orgaddress{\state{Baku}, \country{Azerbaijan}}}

\address[2]{\orgdiv{Department of Mathematics}, \orgname{Shanghai University}, \orgaddress{ \city{Shanghai},  \country{China}}}

\address[3]{\orgname{Azerbaijan National Aviation Academy}, \orgaddress{\state{Baku}, \country{Azerbaijan}}}

\address[4]{ \orgname{Azerbaijan University of Architecture and Construction}, \orgaddress{\state{Baku}, \country{Azerbaijan}}}

\corres{*Elimhan N. Mahmudov, Azerbaijan National Aviation Academy,
Baku, Azerbaijan. \email{elimhan22@yahoo.com}}


\abstract[Summary]{In this paper, we study problems of minimization of a functional
depending on the fractional Caputo derivative of order $0<\alpha
\leq 1$ and the fractional Riemann- Liouville integral of order
$\beta > 0$ at fixed endpoints.  A fractional analogue of the Du
Bois-Reymond lemma is proved, and the Euler-Lagrange conditions
are proved for the simplest problem of fractional variational
calculus with fixed ends and for the fractional isoperimetric
problem. An approach is proposed to obtain the necessary first-order conditions
for the strong and weak extrema, and the necessary optimality
conditions are obtained. From these necessary conditions, as a
consequence, we obtain the Weierstrass condition and its local
modification.

It should be noted that some papers in the literature claim that
the standard proof of the Legendre condition in the classical case
$\alpha=1$ cannot be adapted to the fractional case $0< \alpha <1$
with final constraints. Despite this, we prove the Legendre
conditions by the standard classical method via the Weierstrass
condition. In addition, the necessary Weierstrass-Erdmann
conditions at the corner points are obtained. Examples are
provided to illustrate the significance of the main results obtained.}



\maketitle


\section{Introduction}\label{sec1}
It is known that the main problem of the calculus of variations
arose as a generalization of the brachistochrone problem posed by
I. Bernoulli in 1696. This problem contained typical features of a
new class of mathematical problems. In 1742, Euler, approximating
curves with broken lines, first derived the necessary conditions
in the form of a second-order differential equation that the
extremals had to satisfy. Lagrange later called this equation the
Euler equation. Lagrange himself derived this equation in 1755 by
varying the curve that was supposed to be the extremum. In 1786
and 1837, the necessary conditions of Legendre and Jacobi were
published successively, and in 1879, during his lectures on the
calculus of variations, Weierstrass introduced a fourth necessary
condition. It should be noted that the method of variations
proposed by J.L. Lagrange is the main method of theoretical study
of extremal problems in functional spaces. In this method, the
type of variation used plays a decisive role, since both the
simplicity and the strength of the resulting necessary condition
largely depend on it (see, e.g. \cite{13,17,27}).
Various problems in mechanics and physics played an important role
in the development of the calculus of variations. In turn, the
methods of the calculus of variations are widely used in various
areas of physics, ranging from classical mechanics to the theory
of elementary particles.

Despite the fact that the study of fractional calculus with
derivatives and integrals of noninteger order began more than
three centuries ago, its intensive development began at the end of
the twentieth century \cite{18,19,25,26}. This is primarily due to the
fact that fractional calculus has found application in various
fields of science, technology, as well as in pure and applied
mathematics \cite{3,21,28,35}. Note that fractional differential
equations produce models that are much superior to models using
integer differential equations because they incorporate memory problems into the model.

One of the most remarkable applications of fractional calculus in
physics is found in the context of classical mechanics. In 1996,
Riewe \cite{24} generalized the usual calculus of variations to
Lagrangians depending on fractional derivatives to take into
account the influence of nonconservative forces. There, for the
first time in the literature, the fractional Euler-Lagrange
equation was formulated.
After this work, numerous studies appeared that contain new results
on fractional calculus of variations. As a rule, the works available in the literature are mainly
devoted to the derivation of fractional first-order conditions of Euler-Lagrange type \cite{1,2,6,8,12}.

Note that over the last decade, studies have emerged that focus on various existence  results \cite{27Q,1M,3M,27M,5M,41M,30M,20M} and necessary optimization conditions for both fractional variational and fractional optimal problems \cite{8,31M,14,29,30,31,32,33,34,15M,20MM}. The paper \cite{1M} investigates the existence of solutions to nonlinear fractional differential inclusions \cite{36E} in the sense of Atangana-Baleanu-Caputo fractional derivatives in a Banach space. The study of the main results is based on the set-valued issue of special fixed point theorem combined with Kuratowski’s non-compactness measure. The article \cite{3M}  studies some new nonlinear boundary value problems of Liouville-Caputo type fractional differential equations supplemented with nonlocal multi-point conditions involving lower order fractional derivative. Some well-known fixed-point theory tools are used to establish the existence of solutions to the problems under consideration. In the paper \cite{5M} a new formula for the fractional derivative with the Mittag-Leffler kernel in the form of a series of Riemann-Liouville fractional integrals is established. The existence and uniqueness of results for some families of linear and nonlinear fractional order ordinary differential equations are also proved. The paper \cite{41M} is devoted to the existence and uniqueness of solution for a large class of perturbed sweeping processes formulated by fractional differential inclusions in infinite dimensional setting. The work \cite{30M} investigates a higher-order numerical technique for solving an inhomogeneous time fractional reaction-advection-diffusion equation with a nonlocal condition. In the article \cite{20M} a new concept of the derivative with respect to an arbitrary kernel function is proposed. The proposed concept includes fractional Riemann-Liouville derivatives as special cases. Finally, the existence of a result for a boundary value problem including the introduced derivative operator is proved. 
The aim of the paper \cite{31M} is to study the optimal attainability problem in terms of the fractional $\phi$-Hilfer derivative for a fractional dynamic system. The Euler–Lagrange equations are used to obtain the necessary optimality constraints for the fractional optimal attainability problem. In \cite{15M}, the solvability and optimal control of the impulsive nonlinear fractional evolution Hilfer inclusion with delay in Banach spaces are studied. Using the concept of Clarke subdifferential with delay, the existence of an optimal pair of controls for the Lagrange problem is proved. The article \cite{20MM} studies fractional systems with Riemann–Liouville derivatives. A theorem on the existence and uniqueness of a solution to the fractional ordinary Cauchy problem is presented. Then, Pontryagin's maximum principle is proved for nonlinear fractional control systems with a nonlinear integral performance index. In the paper \cite{27Q} an attempt is made to find a mild solution for a delay control system described by nonlinear fractional evolutionary differential equations in Banach spaces, while being subject to nonlocal conditions. Then sufficient conditions for approximate controllability of a nonlinear fractional control system are investigated. In the study \cite{34} the optimal control of nonlinear systems of fractional order with control delay is considered, where fractional derivatives are expressed in the Caputo sense. For this problem, a necessary condition of first-order optimality is first obtained in the form of Pontryagin's maximum principle. For control that is singular in the sense of Pontryagin's maximum principle, a necessary condition of higher-order optimality is obtained.  
However, much remains to be
done to establish necessary and sufficient conditions for
optimizing fractional variational problems, as well as to study
the existence and regularity of their solutions.  

 It should
be noted that in the early years of studying the fractional calculus
of variations, the integral functional was considered similarly to
the classical calculus of variations. In article \cite{12}, it is noted
that to ensure the existence of solutions in fractional
variational problems, it is necessary to take into account the
Riemann-Liouville fractional integral. Besides, here the integral
functional is considered in the form of a fractional
Riemann-Liouville integral of order $0< \alpha \leq 1$, and the
Euler-Lagrange equation is obtained in integral and differential
form.

It is known that for the simplest problem of the classical
calculus of variations, the Legendre condition usually can be proven in
two ways. The first method is based on the second variation of the
functional, and the second method uses the Weierstrass function.
In the first method, the standard proof of the Legendre condition
in the classical case is based on the existence of a nontrivial
variation $h$ which, together with its ordinary derivative, has a
compact support. However, according to the well-known nonlocal
property of the Caputo fractional derivative of order $0<\alpha <
1$, a nontrivial variation cannot have a compact support together
with its Caputo fractional derivative. This property of the
fractional derivative was referred to by the authors in work \cite{8}
as an obstacle to proving the Legendre condition in the presence
of boundary constraints.
A more general problem is considered in \cite{8}: the problem of
minimizing a general Bolza functional that depends on the Mayer
cost and the fractional Lagrange cost of order $\beta > 0$,
including the fractional Caputo derivative of order $0 < \alpha
\leq 1$, under general mixed initial/final constraints.
In the work \cite{8}, an entire section (Section 3.3) is devoted to
justifying the obstacles in deriving the Legendre condition in a
fractional formulation with boundary constraints.
To overcome this difficulty, they proposed an alternative method
to the classical proof of the Legendre condition. They adapted the
well-known proof of the Pontryagin maximum principle \cite{7}, based on
Ekeland's variational principle \cite{11}, to the problem of fractional
variational calculus with general mixed initial and terminal
constraints. The Euler-Lagrange equation, transversality
conditions, and the Legendre condition were obtained. In their
doctoral dissertation \cite{9}, L. Bourdin argues that a more natural
formulation of the fractional variational problem, which avoids
the pitfalls of studying fractional variational calculus, lies in
choosing $\beta=\alpha$.

After becoming acquainted with these works, we felt the need for a
deeper investigation of similar problems. However, so far we have
not been able to find in the literature an analogue of the proof
of the Legendre condition for the fixed endpoints problem in the
fractional variational case, corresponding to the proof in
classical variational calculus. Moreover, at present, there is no
universally accepted formulation of the Weierstrass condition for
fractional variational problems, and this issue remains open. Our
goal is to fill this gap.

In the present paper, we consider a fractional variational problem
for which the functional is written as a fractional
Riemann-Liouville integral of order $\beta > 0$ and depends on the
fractional Caputo derivative of order $0 < \alpha \leq 1$. We are
looking for solutions to the stated fractional variational problem
in the class of functions for which the functions are continuous
and the fractional Caputo derivatives of order $0< \alpha \leq 1$
are piecewice continuous on a given interval. It should be noted
that the functional and the space of functions are a direct
generalization of the classical case when $\beta=\alpha=1$. As in
the classical case, we first prove a lemma that is a direct
generalization of the classical Du Bois-Reymond lemma. It is shown
here that the relationship between the parameters $\beta$ and
$\alpha$ must be taken into account. This lemma for
$\beta=\alpha=1$ agrees with then classical Du Bois-Reymond lemma,
and for $0 < \beta=\alpha \leq 1$ agrees with the lemma obtained
in article \cite{12}. Using this lemma, the necessary Euler-Lagrange
condition in integral form for a weak local minimum is proven for
the considered problem, where, unlike the Euler-Lagrange condition
obtained in \cite{8}, the relationship between the parameters $\beta$
and $\alpha$ is considered directly.
Using the obtained Euler-Lagrange condition, an analogue of the
first Weierstrass-Erdmann condition for a weak minimum at the
corner point is obtained.
Next, the necessary Euler-Lagrange condition is derived in
integral form for the fractional variational isoperimetric
problem.

It is well known that in the classical calculus of variations the
type of variation used plays a decisive role. Similarly, in the
fractional calculus of variations the type of variation also plays
a significant role. By introducing a special variation depending
on the parameters $a, \ b$ and $k$, we prove necessary first-order
conditions for strong and weak local minima in the case of
nonsmooth Lagrangians with respect to the argument
$^{c}D_{t_{0}+}^{\alpha}x$. Note that from these necessary
conditions, an analogue of the Weierstrass necessary condition and
its local modification follows as a consequence.

Our proof represents a nontrivial adaptation of standard methods
of classical variational calculus to the case of fractional
derivatives. At the same time, it is necessary to overcome a
number of technical difficulties associated with the non-locality
of fractional operators.
From the Weierstrass condition at the corner point we obtain an
analogue of the second Weierstrass-Erdmann  condition.
In addition, we have shown that, as in the classical calculus of
variations, for the simplest problem of fractional calculus of
variations, the Legendre conditions can be obtained from the
Weierstrass condition. A separate article is devoted to the proof
of the Legendre condition using the second variation of the
functional.

Thus, the main novetly of this work is  one of the possible
approaches to proving necessary conditions of first-and
second-order optimality in a fractional-variational problem. The
necessary Euler-Lagrange condition in integral form, the necessary
condition for the isoperimetric problem, the necessary first-order
conditions for strong and weak local minima in the case of
non-smooth Lagrangians in the argument $^{c}D_{t_{0}+}^{\alpha}x$,
as well as the Weierstrass and Legendre  conditions are obtained.

The rest of the paper is organized as follows.

In Section \ref{sec2}, the
definitions and basic properties of fractional order integrals and
derivatives are recalled, and also some preliminary results.
In Section \ref{sec3}, the fundamental Du Bois-Reymond lemma for fractional
variational calculus is proved.
In Section \ref{sec4}, the Euler-Lagrange and Weierstrass-Erdmann
conditions for weak minima are proved.
In Section \ref{sec5}, the  necessary Euler-Lagrange condition for the
fractional variational isoperimetric problem is derived.
In Section \ref{sec6}, a necessary first-order condition for a strong and
weak local minimum in the fractional problem is derived, and then,
as a consequence, the Weierstrass, Legendre, and
Weierstrass-Erdmann conditions are obtained.

\section{Notations  and basics from fractional calculus}\label{sec2}
This section is devoted to recall basic definitions and results
about Riemann-Liouville and Caputo fractional operators. All of
the presentations below is standard and mostly extracted from the
monographs \cite{19,26}.

Let $\mathbb{R}^{n}$  be the space of $n$- dimensional vectors and
$\|\cdot\|$ be the norm in this space. Let the numbers $t_{0}, \
t_{1} \in \mathbb{R}$, $t_{0}<t_{1}$ be fixed.

Let $C([t_{0}, t_{1}], \mathbb{R}^{n})$ be the space of continuous
functions on $[t_{0}, t_{1}]$ with values in $\mathbb{R}^{n}$,
endowed with the uniform norm $\|\cdot\|_{C}.$  $L^{1}([t_{0},
t_{1}], \mathbb{R}^{n})$ is the Lebesgue space of summable functions
defined on $[t_{0}, t_{1}]$ with values in $\mathbb{R}^{n}$,
endowd with its usual norm $\|\cdot\|_{L^{1}}.$
$L^{\infty}([t_{0}, t_{1}], \mathbb{R}^{n})$ is the Lebesgue space of
essentially bounded functions defined on $[t_{0}, t_{1}]$ with
values in $\mathbb{R}^{n}$, endowed with its usual norm
$\|\cdot\|_{L^{\infty}}.$

\begin{definition}\label{def21} Let $\alpha >0.$ For a function $\varphi
\in L^{1}([t_{0}, t_{1}], \mathbb{R}^{n})$ the left-sided and
right-sided Riemann-Liouville fractional integrals of the order
$\alpha$ are defined for a.e. $t \in [t_{0}, t_{1}]$ by
$$
(I_{t_{0}+}^{\alpha}\varphi)(t)=\frac{1}{\Gamma(\alpha)}\int\limits_{t_{0}}^{t}(t-\tau)^{\alpha-1}\varphi(\tau)d\tau\,\,\,
\texttt{and} \,\,\,
(I_{t_{1}-}^{\alpha}\varphi)(t)=\frac{1}{\Gamma(\alpha)}\int\limits_{t}^{t_{1}}(\tau-t)^{\alpha-1}\varphi(\tau)d\tau,
$$
respectively. For $\alpha=0$ we set
$I^{0}_{t_{0}+}\varphi=I_{t_{1}-}^{0}\varphi=\varphi$.

If $\varphi \in L^{\infty}([t_{0}, t_{1}], \mathbb{R}^{n}),$ then
the above functions are defined and finite everywhere on $[t_{0},
t_{1}]$.
\end{definition}
\begin{definition}\label{def22} Let $\alpha \in (0, \ 1)$ and $\varphi
\in L^{1}([t_{0}, t_{1}], \mathbb{R}^{n}).$ For a function
$\varphi$ the left-sided and right-sided Riemann-Liouville
fractional derivatives of the order $\alpha$ are defined for a.e.
$t \in [t_{0}, t_{1}]$ by
$$
(D_{t_{0}+}^{\alpha}\varphi)(t)=\frac{d}{dt}(I_{t_{0}+}^{1-\alpha}\varphi)(t)=
\frac{1}{\Gamma(1-\alpha)}\frac{d}{dt}
\int\limits_{t_{0}}^{t}(t-\tau)^{-\alpha}\varphi(\tau)d\tau\,
$$
and
$$
(D_{t_{1}-}^{\alpha}\varphi)(t)=-\frac{d}{dt}(I_{t_{1}-}^{1-\alpha}\varphi)(t)=
-\frac{1}{\Gamma(1-\alpha)}\frac{d}{dt}
\int\limits_{t}^{t_{1}}(\tau-t)^{-\alpha}\varphi(\tau)d\tau,
$$
if $(I_{t_{0}+}^{1-\alpha}\varphi)(\cdot)$ and
$(I_{t_{1}-}^{1-\alpha}\varphi)(\cdot)$ are absolutely continuous
functions on $[t_{0}, t_{1}]$.
\end{definition}
\begin{definition}\label{def23} Let $\alpha \in (0, \ 1),$ and $\varphi
\in C([t_{0}, t_{1}], \mathbb{R}^{n}).$ For a function $\varphi$
the left-sided and right-sided Caputo  fractional derivatives of
the order $\alpha$ are defined for a.e. $t \in [t_{0}, t_{1}]$ by
$$
(^{c}D_{t_{0}+}^{\alpha}\varphi)(t)=\frac{d}{dt}(I_{t_{0}+}^{1-\alpha}(\varphi(\cdot)-\varphi(t_{0})))(t)=
\frac{1}{\Gamma(1-\alpha)}\frac{d}{dt}
\int\limits_{t_{0}}^{t}(t-\tau)^{-\alpha}(\varphi(\tau)-\varphi(t_{0}))d\tau\,
$$
and
$$
(^{c}D_{t_{1}-}^{\alpha}\varphi)(t)=-\frac{d}{dt}(I_{t_{1}-}^{1-\alpha}(\varphi(\cdot)-\varphi(t_{1})))(t)=
-\frac{1}{\Gamma(1-\alpha)}\frac{d}{dt}
\int\limits_{t}^{t_{1}}(\tau-t)^{-\alpha}(\varphi(\tau)-\varphi(t_{1}))d\tau,
$$
if $(I_{t_{0}+}^{1-\alpha}(\varphi(\cdot)-\varphi(t_{0})))(\cdot)$
and
$(I_{t_{1}-}^{1-\alpha}(\varphi(\cdot)-\varphi(t_{1})))(\cdot)$
are absolutely continuous functions on $[t_{0}, t_{1}]$.

By $PC([t_{0}, t_{1}], \mathbb{R}^{n})$ we denote the set of
piecewise continuous functions that are continuous everywhere on
$[t_{0}, \ t_{1}]$, except for a finite number of points
$\sigma_{i}\in (t_{0}, t_{1}),$ $i=\overline{1, \ k},$ and hawing first kind of discontinuities at these points.

In what follows, the set of continuity points of a piecewise
continuous function will be denoted by $T$, and the set of points
of discontinuity of the first kind by A.

By $PC^{\alpha}([t_{0}, t_{1}], \mathbb{R}^{n}),$ $0<\alpha \leq
1,$ we denote the set of all functions $\varphi \in C([t_{0},
t_{1}], \mathbb{R}^{n}),$ such that
$$
\varphi(t)=\varphi(t_{0})+(I_{t_{0}+}^{\alpha}\psi)(t), \,\,\, t
\in [t_{0}, t_{1}],
$$
with $\psi \in PC([t_{0}, t_{1}], \mathbb{R}^{n}).$
\end{definition}
\begin{proposition}\label{pro21} For any $\varphi \in PC^{\alpha}([t_{0},
t_{1}], \mathbb{R}^{n}),$ $0<\alpha \leq 1,$ the value
$(^{c}D_{t_{0}+}^{\alpha}\varphi)(t)$ is correctly defined for
every $t \in T.$ Moreover, the inclusion
$(^{c}D_{t_{0}+}^{\alpha}\varphi)\in PC([t_{0}, t_{1}],
\mathbb{R}^{n})$ holds (i.e., there exists $\psi \in PC([t_{0},
t_{1}], \mathbb{R}^{n})$ such that
$\psi(t)=(^{c}D_{t_{0}}^{\alpha}\varphi)(t)$ for every $t \in T$)
and
$$
(I_{t_{0}+}^{\alpha}(^{c}D_{t_{0}+}^{\alpha}\varphi))(t)=\varphi(t)-\varphi(t_{0}),
\,\,\, t \in [t_{0}, t_{1}].
$$
\end{proposition}
\begin{proposition}\label{pro22} Let $\alpha \in (0, \ 1)$ and $\varphi
\in PC([t_{0}, t_{1}], \mathbb{R}^{n}),$ then
$(^{c}D_{t_{0}+}^{\alpha}(I_{t_{0}+}^{\alpha}\varphi))(t)=\varphi(t)$
and
$(^{c}D_{t_{1}-}^{\alpha}(I_{t_{1}-}^{\alpha}\varphi))(t)=\varphi(t)$
for every $t \in T.$

We also define the set $PC_{0}^{\alpha}([t_{0}, t_{1}],
\mathbb{R}^{n})=\{\varphi| \varphi \in PC^{\alpha}([t_{0}, t_{1}],
\mathbb{R}^{n}), \ \varphi(t_{0})=\varphi(t_{1})=0\}.$

By $\langle u, \ v \rangle$ we denote the scalar product of
vectors $u$ and $v$.
\end{proposition}
\section{Fundamental lemma of fractional calculus of variations}\label{sec3}

In this section we prove the generalized Du Bois-Reymond lemma,
which plays an essential role in deriving the first-order
necessary condition.
\begin{lemma}\label{lem31} Suppose that $f \in PC([t_{0},
t_{1}],\mathbb{R}), \,$ $0<\alpha \leq 1$ and $\beta > 0.$ Then to
fulfill equality
\begin{equation}\label{equ1}
\int\limits_{t_{0}}^{t_{1}}(t_{1}-t)^{\beta-1}f(t)(^{c}D_{t_{0}+}^{\alpha}h)(t)dt=0,
\,\,\,  h \in PC_{0}^{\alpha}([t_{0}, t_{1}], \mathbb{R})
\end{equation}%
it is necessary and sufficient that

1) $(t_{1}-t)^{\beta-\alpha}f(t)=0$ on $t \in T,$ if $\beta
> \alpha>0,$

2) $f(t)=\frac{k}{\Gamma(\alpha)}(t_{1}-t)^{\alpha-\beta}$ on $t
\in T,$ if $0 < \beta \leq \alpha \leq1$ for some constant $k\neq
0.$
\end{lemma}
\begin{proof} \textbf{Sufficiency.} The equality (\ref{equ1}) follows immediately from $f(t)=0,$ $t \in T$. Let now conditions 2) be satisfied. Then
$$
\int\limits_{t_{0}}^{t_{1}}(t_{1}-t)^{\beta-1}f(t)(^{c}D_{t_{0}+}^{\alpha}h)(t)dt=
\frac{k}{\Gamma(\alpha)}\int\limits_{t_{0}}^{t_{1}}(t_{1}-t)^{\beta-1}(t_{1}-t)^{\alpha-\beta}(^{c}D_{t_{0}+}^{\alpha}h)(t)dt
$$
$$
=\frac{k}{\Gamma(\alpha)}\int\limits_{t_{0}}^{t_{1}}(t_{1}-t)^{\alpha-1}(^{c}D_{t_{0}+}^{\alpha}h)(t)dt=k(h(t_{1})-h(t_{0}))=0.
$$

\textbf{Necessity.} Let $f \in PC([t_{0}, \ t_{1}], \mathbb{R}),$ $0
< \alpha \leq 1,$ $\beta>0$ and equality (\ref{equ1}) is satisfied. First
consider the case $\beta > \alpha$. Let's put
$$
h(t)=\frac{1}{\Gamma(\alpha)}\int\limits_{t_{0}}^{t}(t-\tau)^{\alpha-1}
[\Gamma(\alpha)(t_{1}-\tau)^{\beta-\alpha}f(\tau)-k_{0}]d\tau,
\,\, t \in [t_{0}, \ t_{1}],
$$
where $k_{0}$ is a real number.

The fractional Caputo derivative of the function $h$ has the form
$$
(^{c}D_{t_{0}+}^{\alpha}h)(t)=\Gamma(\alpha)(t_{1}-t)^{\beta-\alpha}f(t)-k_{0},
\,\, t \in T.
$$

Therefore, the function $h$ belongs to the set $PC^{\alpha}([t_{0},
\ t_{1}], \mathbb{R})$ and $h(t_{0})=0.$ If you choose $k_{0}$ in
the form:
$$
k_{0}=\frac{\Gamma(\alpha+1)}{(t_{1}-t_{0})^{\alpha}}\int\limits_{t_{0}}^{t_{1}}(t_{1}-t)^{\beta-1}f(t)dt,
$$
then it is obvious that
$$
h(t_{1})=\frac{1}{\Gamma(\alpha)}\int\limits_{t_{0}}^{t_{1}}(t_{1}-\tau)^{\alpha-1}
[\Gamma(\alpha)(t_{1}-\tau)^{\beta-\alpha}f(\tau)-k_{0}]d\tau=0.
$$

Thus, for a given $k_{0}$, the function $h$ belongs to the set
$PC_{0}^{\alpha}([t_{0},t_{1}], \mathbb{R}).$ For the chosen $h$ we
write equality (\ref{equ1}) in the form
$$
\frac{1}{\Gamma(\alpha)}\int\limits_{t_{0}}^{t_{1}}(t_{1}-t)^{\alpha-1}
\Gamma(\alpha)(t_{1}-t)^{\beta-\alpha}f(t)[\Gamma(\alpha)(t_{1}-t)^{\beta-\alpha}f(t)
-k_{0}]dt=0.
$$

On the other hand
$$
\frac{k_{0}}{\Gamma(\alpha)}\int\limits_{t_{0}}^{t_{1}}(t_{1}-t)^{\alpha-1}
[\Gamma(\alpha)(t_{1}-t)^{\beta-\alpha}f(t) -k_{0}]dt=0.
$$

From the last two equalities we have
$$
\frac{1}{\Gamma(\alpha)}\int\limits_{t_{0}}^{t_{1}}(t_{1}-t)^{\alpha-1}
[\Gamma(\alpha)(t_{1}-t)^{\beta-\alpha}f(t) -k_{0}]^{2}dt=0.
$$

It follows that
$$
\Gamma(\alpha)(t_{1}-t)^{\beta-\alpha}f(t)=k_{0}, \,\,\, t \in T.
$$

Taking into account the condition $\beta> \alpha$, from this
equality for $t=t_{1}$ we obtain that $k_{0}=0$. Then for any $t
\in T$ we have $(t_{1}-t)^{\beta-\alpha}f(t)=0.$

Now consider the case $0 < \beta \leq \alpha \leq 1.$ For $h \in
PC_{0}^{\alpha}([t_{0},{t_1}], \mathbb{R})$ we have $h(t_{0})=0$ and
\begin{equation}\label{equ2}
h(t_{1})=\frac{1}{\Gamma(\alpha)}\int\limits_{t_{0}}^{t_{1}}(t_{1}-t)^{\alpha-1}
(^{c}D^{\alpha}_{t_{0}+}h)(t)dt=0.
\end{equation}%

Multiplying equality (\ref{equ2}) by $0 \neq k \in \mathbb{R}$ and
subtracting from equality (\ref{equ1}), we get
\begin{equation}\label{equ3}
\int\limits_{t_{0}}^{t_{1}}(t_{1}-t)^{\beta-1}\left[f(t)-\frac{k}{\Gamma(\alpha)}(t_{1}-t)^{\alpha-\beta}
\right](^{c}D^{\alpha}_{t_{0}+}h)(t)dt=0.
\end{equation}%

Now we choose a constant number $k$ so that equation
\begin{equation}\label{equ4}
(^{c}D^{\alpha}_{t_{0}+}h)(t)=f(t)-\frac{k}{\Gamma(\alpha)}(t_{1}-t)^{\alpha-\beta},
\,\, t \in T,
\end{equation}%
in the space $PC_{0}^{\alpha}([t_{0}, t_{1}], \mathbb{R})$ has a
unique solution.

Obviously, the solution to equation (\ref{equ4}), satisfying the condition
$h(t_{0})=0$, has the form
$$
h(t)=\frac{1}{\Gamma(\alpha)}\int\limits_{t_{0}}^{t}(t-\tau)^{\alpha-1}\left(f(\tau)-\frac{k}{\Gamma(\alpha)}(t_{1}-\tau)^{\alpha-\beta}
\right)d\tau, \,\,\, t \in [t_{0}, \ t_{1}].
$$

From the conditions
$$
h(t_{1})=\frac{1}{\Gamma(\alpha)}\int\limits_{t_{0}}^{t_{1}}(t_{1}-\tau)^{\alpha-1}\left(f(\tau)-\frac{k}{\Gamma(\alpha)}(t_{1}-\tau)^{\alpha-\beta}
\right)d\tau=0,
$$
we find
$$
k=\frac{(2\alpha-\beta)\Gamma(\alpha)}{(t_{1}-t_{0})^{2\alpha-\beta}}\int\limits_{t_{0}}^{t_{1}}(t_{1}-t)^{\alpha-1}f(t)dt.
$$

Taking into account (\ref{equ4}) of (\ref{equ3}) we have
$$
\int\limits_{t_{0}}^{t_{1}}(t_{1}-t)^{\beta-1}\left[f(t)-\frac{k}{\Gamma(\alpha)}(t_{1}-t)^{\alpha-\beta}
\right]^{2}dt=0.
$$
If follows that  $ \;f(t)=\frac{k}{\Gamma(\alpha)}(t_{1}-t)^{\alpha-\beta}, \,\,\, t
\in T,$ if\, $ 0< \beta \leq \alpha \leq 1. 
$ \end{proof}

\section{The Euler-Lagrange necessary condition for the simplest
problem of fractional variational calculus}\label{sec4}

\bigskip
 For $0<\alpha \leq 1$ and $\beta>0$ we call the simplest problem
 of fractional calculus of variations the following extremal
 problem in the space $PC^{\alpha}([t_{0}, t_{1}], \mathbb{R}^{n})$:
$$
J(x(\cdot))=\int\limits_{t_{0}}^{t_{1}}(t_{1}-t)^{\beta-1}L(t,
x(t), (^{c}D_{t_{0}+}^{\alpha}x)(t))dt \rightarrow extr, \\
x(t_{0})=x_{0}, \,\,\, x(t_{1})=x_{1}. \eqno(P)
$$

Here the segment $[t_{0}, t_{1}]$ is assumed to be fixed and finite,
$t_{0}<t_{1}.$ $L=L(t, x, y)$ is a function $2n+1$ variables, called
the integrand. The extremum in problem (P) is considered among the
space functions $PC^{\alpha}([t_{0}, t_{1}],\mathbb{R}^{n})$
satisfying the conditions at the ends or boundary conditions:
$x(t_{0})=x_{0}$, $x(t_{1})=x_{1}.$ Such functions are called
admissible.

\begin{definition}\label{def41} We will say that an admissible function
$x^{0}$ deliver a weak local minimum in problem (P), and write
$x^{0} \in wlocmin P$ if there exists $\delta>0$ such that
$J(x(\cdot))\geq J(x^{0}(\cdot))$ for any admissible function $x$
for which
$$
\|x(\cdot)-x^{0}(\cdot)\|_{PC^{\alpha}([t_{0},t_{1}],
\mathbb{R}^{n})}< \delta,
$$
where
$$
\|y\|_{PC^{\alpha}([t_{0},t_{1}],
\mathbb{R}^{n})}=\|y(\cdot)\|_{C([t_{0},
t_{1}],\mathbb{R}^{n})}+\|(^{c}D_{t_{0}+}^{\alpha}y)(\cdot)\|_{L^{\infty}([t_{0},
t_{1}], \mathbb{R}^{n})}.
$$
\end{definition}
\begin{theorem}\label{teo41} Let $0< \alpha \leq 1,$ $\beta>0$ and the
function $x^{0}$ deliver a weak local minimum in the problem (P)
$(x^{0} \in wlocmin P),$ the functions $L, \ L_{x}, \ L_{y}$ are
continuous in some neighborhood of the graph $\{(t, x^{0}(t),
(^{c}D_{t_{0}+}^{\alpha}x^{0})(t))|t \in [t_{0}, t_{1}]\}.$ Then
the function $x^{0}$ satisfies the equation:

1)
\begin{equation}\label{equ5}
(t_{1}-t)^{1-\alpha}(I_{t_{1}-}^{\alpha}b)(t)+(t_{1}-t)^{\beta-\alpha}L_{y}(t,
x^{0}(t),(^{c}D_{t_{0}+}^{\alpha}x^{0})(t))=0 
\end{equation}%
on $T,$ if $0<\alpha< \beta,$

2)
\begin{equation}\label{equ6}
(t_{1}-t)^{1-\beta}(I_{t_{1}-}^{\alpha}b)(t)+L_{y}(t,
x^{0}(t),(^{c}D_{t_{0}+}^{\alpha}x^{0})(t))=\frac{k}{\Gamma(\alpha)}(t_{1}-t)^{\alpha-\beta}
\end{equation}%
on $T,$  if \, $0<\beta \leq \alpha\leq 1,$ for some constant
$0\neq k \in \mathbb{R}^{n},$  where $b(t)=(t_{1}-t)^{\beta-1}$
$L_{x}(t, x^{0}(t), (^{c}D_{t_{0}+}^{\alpha}x^{0})(t))$.
\end{theorem}
\begin{proof} Let's take an arbitrary fixed function $h \in
PC_{0}^{\alpha}([t_{0}, t_{1}], \mathbb{R}^{n}).$ Since $x^{0} \in
wlocmin P$ then the function of one variable
$$
\phi(\lambda)=J(x^{0}(\cdot)+\lambda
h(\cdot))=\int\limits_{t_{0}}^{t_{1}}(t_{1}-t)^{\beta-1}L(t,
x^{0}(t)+\lambda h(t),
(^{c}D_{t_{0}+}^{\alpha}x^{0})(t)+\lambda(^{c}D_{t_{0}+}^{\alpha}h)(t))dt
$$
has an extremum at $\lambda=0$. From the smoothness conditions
imposed on $L, \, x^{0}, \ h,$ it follows that the function $\phi$
is differentiable at zero. But then, according to Fermat's
theorem, $\phi'(0)=0.$ Differentiating the function $\phi$ and
assuming $\lambda=0$ we obtain
$$
\phi'(0)=\int\limits_{t_{0}}^{t_{1}}(t_{1}-t)^{\beta-1} [\langle
L_{x}(t, x^{0}(t), (^{c}D_{t_{0}+}^{\alpha}x^{0})(t)), h(t)
\rangle
$$
\begin{equation}\label{equ7}
+\langle L_{y}(t, x^{0}(t), (^{c}D_{t_{0}+}^{\alpha}x^{0})(t)),
(^{c}D_{t_{0}+}^{\alpha}h)(t)\rangle]dt=0,\,\,\, h \in
PC_{0}^{\alpha}([t_{0}, t_{1}], \mathbb{R}^{n}).      
\end{equation}%

Now we put $h(t)=\varphi(t)r,$ where $\varphi \in
PC_{0}^{\alpha}([t_{0}, t_{1}], \mathbb{R})$ and $r \in
\mathbb{R}^{n}.$ Given the representation
$$
\varphi(t)=\frac{1}{\Gamma(\alpha)}
\int\limits_{t_{0}}^{t}(t-\tau)^{\alpha-1}(^{c}D_{t_{0}+}^{\alpha}\varphi)(\tau)d\tau,
$$
we transform (\ref{equ7}) as follows:
$$
\int\limits_{t_{0}}^{t_{1}}(t_{1}-t)^{\beta-1}[\langle L_{x}(t,
x^{0}(t), (^{c}D_{t_{0}+}^{\alpha}x^{0})(t)), r
\rangle\frac{1}{\Gamma(\alpha)}\int\limits_{t_{0}}^{t}(t-\tau)^{\alpha-1}(^{c}D_{t_{0}+}^{\alpha}\varphi)(\tau)d\tau
$$
$$
+\langle L_{y}(t, x^{0}(t), (^{c}D_{t_{0}+}^{\alpha}x^{0})(t)), r
\rangle(^{c}D_{t_{0}+}^{\alpha}\varphi)(t)]dt
$$
$$
=\int\limits_{t_{0}}^{t_{1}}(t_{1}-t)^{\beta-1}
\langle\frac{(t_{1}-t)^{1-\beta}}{\Gamma(\alpha)}\int\limits_{t}^{t_{1}}(t_{1}-\tau)^{\beta-1}(\tau-t)^{\alpha-1}L_{x}(\tau,
x^{0}(\tau),(^{c}D_{t_{0}+}^{\alpha}x^{0})(\tau))d\tau
$$
$$
+L_{y}(t, x^{0}(t), (^{c}D_{t_{0}+}^{\alpha}x^{0})(t)), r \rangle
(^{c}D_{t_{0}+}^{\alpha}\varphi)(t)dt=0.
$$

Then by virtue of Lemma \ref{lem31} we have

1)
$\langle(t_{1}-t)^{1-\alpha}(I_{t_{1}-}^{\alpha}b)(t)+(t_{1}-t)^{\beta-\alpha}L_{y}(t),
r \rangle=0$ on $t \in T,$ if $\beta> \alpha >0,$

2) $\langle(t_{1}-t)^{1-\beta}(I_{t_{1}-}^{\alpha}b)(t)+L_{y}(t),
r \rangle=\frac{k_{0}}{\Gamma(\alpha)}t_{1}-t)^{\alpha-\beta}$ on
$t \in T,$ if $0< \beta\leq \alpha \leq 1,$ where $k_{0}\in
\mathbb{R}$ is a constant, $b(t)=(t_{1}-t)^{\beta-1}L_{x}(t,
x^{0}(t), (^{c}D_{t_{0}+}^{\alpha}x^{0})(t)).$

Since $r \in \mathbb{R}^{n}$ is arbitrary, the assertions of
Theorem \ref{teo41} follow from Lemma \ref{lem31}. 
\end{proof}

Equations (\ref{equ5}) and (\ref{equ6}) represent the fractional analogue of the
Euler-Lagrange integral equation. Note that the point of
discontinuity of the first kind of the derivative of a given
function is called a corner point of that function.

\begin{corollary}\label{cor41}\textbf{(The first Weierstrass-Erdmann condition)}
Let $0<\alpha \le 1,$ $\beta
>0$ and the vector function $x^0$ provide the problem (P)  with
at least a weak local extremum, then at each corner point $t\in A$
the equality

\begin{equation}\label{equ8}
f_y\left(t,x^0\left(t\right),\left({}^c{D^{\alpha
}_{t_0+}}x^0\right)\left(t-0\right)\right)=f_y\left(t,x^0\left(t\right),\left({}^c{D^{\alpha
}_{t_0+}}x^0\right)\left(t+0\right)\right)
\end{equation}%
is satisfied.
\end{corollary}
\begin{proof} We show that equality (\ref{equ8}) is a direct consequence
of the Euler-Lagrange equation. Indeed, in the case $\beta >\alpha
>0$ the equality

\[{(t_1-t)}^{\beta -\alpha }L_y\left(t,x^0\left(t\right),\left({}^c{D^{\alpha }_{t_0+}}x^0\right)\left(t\right)\right) \]

\begin{equation}\label{equ9}
=\frac{-{(t_1-t)}^{1-\alpha }}{\Gamma (\alpha )}
\int\limits^{t_1}_t{{\left(t_1-\tau \right)}^{\beta -1}{\left(\tau
-t\right)}^{\alpha -1}L_x(\tau ,x^0\left(\tau \right),\ }\
\left({}^c{D^{\alpha }_{t_0+}}x^0\right)\left(\tau \right))\ d\tau
,\ \ t\in T , 
\end{equation}%
and the case $0<\beta \le \alpha \le 1$ the equality
\[L_y\left(t,x^0\left(t\right),\left({}^c{D^{\alpha }_{t_0+}}x^0\right)\left(t\right)\right)=-\frac{{\left(t_1-t\right)}^{1-\beta }}
{\Gamma \left(\alpha \right)}\]
\begin{equation}\label{equ10}
\times \int\limits^{t_1}_t{{\left(t_1-\tau \right)}^{\beta
-1}{\left(\tau -t\right)}^{\alpha -1}L_x(\tau ,x^0\left(\tau
\right),\ }\ \left({}^c{D^{\alpha }_{t_0+}}x^0\right)\left(\tau
\right))\ d\tau +\frac{k}{\Gamma (\alpha )}{(t_1-t)}^{\alpha
-\beta },\ \ t\in T, 
\end{equation}%
are satisfied, respectively.

Thus, the function on the right-hand side of equality (\ref{equ9}) is
continuous on the interval $\left[t_0,t_1\right].$ Therefore, the
function on the left-hand side is the restriction of some function
from $C(\left[t_0,t_1\right],\ \mathbb{R}^n)$ to the set $T$. It
follows that
\[{(t_1-(t+0))}^{\beta -\alpha }L_y\left(t,x^0\left(t\right),\left({}^c{D^{\alpha }_{t_0+}}
x^0\right)\left(t+0\right)\right)\]\[={(t_1-(t-0))}^{\beta -\alpha
}L_y\left(t,x^0\left(t\right),\left({}^c{D^{\alpha
}_{t_0+}}x^0\right)\left(t-0\right)\right),\ \ \ t\in A.\]

From the continuity of the function ${(t_1-t)}^{\beta -\alpha }$
on the segment $\left[t_0,t_1\right]$ follows equality (\ref{equ8}). Using
a similar method, we can prove the validity of equality (\ref{equ8}) based
on equality (\ref{equ10}).
\end{proof}

\bigskip

\section{The Fractional Isoperimetric Problem}\label{sec5}

\bigskip

We now consider the following isoperimetric problem of the
fractional calculus of variations. Given a functional
\begin{equation}\label{equ11}
J_{0}(x(\cdot))=\int\limits_{t_{0}}^{t_{1}}(t_{1}-t)^{\beta-1}L_{0}(t,
x(t), (^{c}D_{t_{0}+}^{\alpha}x)(t))dt, 
\end{equation}%
which functions $x$ minimize (or maximize)$J_{0}$, when subject to
given boundary conditions
\begin{equation}\label{equ12}
x(t_{0})=x_{0}, \,\,\,\, x(t_{1})=x_{1}, 
\end{equation}%
and an integral constraint
\begin{equation}\label{equ13}
J_{1}(x(\cdot))=\int\limits_{t_{0}}^{t_{1}}(t_{1}-t)^{\beta-1}L_{1}(t,
x(t), (^{c}D_{t_{0}+}^{\alpha}x)(t))dt=l_{1},  
\end{equation}%
where $0<\alpha \leq 1,$ $\beta>0$ and $l_{1} \in \mathbb{R}$
given numbers.

The segment $[t_{0}, t_{1}]$ is fixed and finite, $t_{0}< t_{1}.$
Constraints of type (\ref{equ13}) are called isoperimetric. The extremum in
the problem is considered among the functions $x \in
PC^{\alpha}([t_{0}, t_{1}], \mathbb{R})$ satisfying the
isoperimetric condition (\ref{equ13}) and conditions (\ref{equ12}) at the ends; such
functions are called admissible.

In the following presentations, we denote problem (\ref{equ11})-(\ref{equ13}) by $P_1$.
\begin{definition}\label{def51} We say that an admissible function
$x^{0}$ delivers a weak local minimum in problem $P_1$ and
write $x^{0}\in wlocmin P_1$ if there exists $\delta >0$ such that
$J_{0}(x(\cdot))\geq J_{0}(x^{0}(\cdot))$ for any  admissible
function $x$, for which
$\|x(\cdot)-x^{0}(\cdot)\|_{PC^{\alpha}([t_{0}, t_{1}],
\mathbb{R})}< \delta.$
\end{definition}
\begin{theorem}\label{teo51} Let $0<\alpha \leq 1,$ $\beta>0$ and the
function $x^{0}$ delivers a weak local minimum in problem $P_1$ $(x^{0}\in wlocmin P_1),$ the functions $L_{i}, \ L_{ix},
\ L_{iy},$ $i=0, \ 1,$- are continuous in some neighborhood of the
graph $\{(t, x^{0}(t), (^{c}D_{t_{0}+}^{\alpha}x^{0})(t)),$ $t \in
[t_{0}, t_{1}]\}.$ Then there will be numbers $\mu_{0}, \ \mu_{1}$
not equal to zero at the same time and such that for the
Lagrangian $L=\mu_{0}L_{0}+\mu_{1}L_{1}$ the following equation
holds:

1) If $0<\alpha< \beta,$ then
\begin{equation}\label{equ14}
(t_{1}-t)^{1-\alpha}(I_{t_{1}-}^{\alpha}b)(t)+(t_{1}-t)^{\beta-\alpha}L_{y}(t,
x^{0}(t), (^{c}D_{t_{0}+}^{\alpha}x^{0})(t))=0, \,\,\, t \in T,
\end{equation}%

2) If $0< \beta \leq \alpha\leq 1,$ then
\begin{equation}\label{equ15}
(t_{1}-t)^{1-\beta}(I_{t_{1}-}^{\alpha}b)(t)+L_{y}(t, x^{0}(t),
(^{c}D_{t_{0}+}^{\alpha}x^{0})(t))=\frac{k}{\Gamma(\alpha)}(t_{1}-t)^{\alpha-\beta},
\,\,\, t \in T,
\end{equation}%
where $k$ is some constant, and $b(t)=(t_{1}-t)^{\beta-1}L_{x}(t,
x^{0}(t), (^{c}D_{t_{0}+}^{\alpha}x^{0})(t)).$
\end{theorem}
\begin{proof} For an arbitrary but fixed function $h \in
PC_{0}^{\alpha}([t_{0}, t_{1}], \mathbb{R})$, we calculate the
first variations of the functionals $J_{0}$ and $J_{1}$ by
Lagrange
$$
\delta J_{i}(x^{0}(\cdot),
h(\cdot))=\int\limits_{t_{0}}^{t_{1}}(t_{1}-t)^{\beta-1}[L_{ix}(t,
x^{0}(t), (^{c}D_{t_{0}+}^{\alpha}x^{0})(t))h(t)
$$
$$
+L_{iy}(t, x^{0}(t),
(^{c}D_{t_{0}+}^{\alpha}x^{0})(t))(^{c}D_{t_{0}+}^{\alpha}h)(t)]dt,
\,\,\, i=0, \ 1.
$$

One of two cases is possible: either $\delta J_{1}(x^{0}(\cdot),
h(\cdot)) \equiv 0,$ $\forall h \in PC_{0}^{\alpha}([t_{0},
t_{1}], \mathbb{R}),$ or there is a function $h_{1} \in
PC_{0}^{\alpha}([t_{0}, t_{1}], \mathbb{R})$ such that $\delta
J_{1}(x^{0}(\cdot), h_{1}(\cdot)) \neq 0.$
\\

In the first case, according to Theorem \ref{teo41} it follows that the
following equations are satisfied:

1) If $0<\alpha <\beta,$ then
$$
(t_{1}-t)^{1-\alpha}(I_{t_{1}-}^{\alpha}b_{1})(t)+(t_{1}-t)^{\beta-\alpha}L_{1y}(t,
x^{0}(t), (^{c}D_{t_{0}+}^{\alpha}x^{0})(t))=0, \,\,\, t \in T,
$$

2) If $0<\beta \leq \alpha \leq 1,$ then
$$
(t_{1}-t)^{1-\beta}(I_{t_{1}-}^{\alpha}b_{1})(t)+L_{1y}(t,
x^{0}(t),
(^{c}D_{t_{0}+}^{\alpha}x^{0})(t))=\frac{k}{\Gamma(\alpha)}(t_{1}-t)^{\alpha-\beta},
\,\,\, t \in T,
$$
where $k \neq 0$ is some constant, $
b_{1}(t)=(t_{1}-t)^{\beta-1}L_{1x}(t, x^{0}(t),
(^{c}D_{t_{0}+}^{\alpha}x^{0})(t)).$

Putting $\mu_{0}=0, \ \mu_{1}=1$ from these equations we
immediately obtain equations (\ref{equ14}) and (\ref{equ15}).

Now let's consider the second case. To do this, we introduce
functions of two variables
$$
\Gamma_{i}(\gamma_{0},
\gamma_{1})=J_{i}(x^{0}(\cdot)+\gamma_{0}h(\cdot)+\gamma_{1}h_{1}(\cdot)),
\,\,\, i=0, \ 1,
$$
that are continuously differentiable in the  neighborhood of zero, and
$$
\frac{\partial \Gamma_{i}(0, 0)}{\partial \gamma_{0}}=\delta
J_{i}(x^{0}(\cdot), h(\cdot)), \,\,\,\,\frac{\partial
\Gamma_{i}(0, 0)}{\partial \gamma_{1}}=\delta J_{i}(x^{0}(\cdot),
h_{1}(\cdot)), \,\,\, i=0, \ 1.
$$

We show that for any $h \in PC_{0}^{\alpha}([t_{0}, t_{1}],
\mathbb{R})$ the equality
\begin{equation}\label{equ16}
\frac{\partial(\Gamma_{0}, \Gamma_{1})}{\partial (\gamma_{0},
\gamma_{1})}=\left| \begin{array}{c} \delta J_{0}(x^{0}(\cdot),
h(\cdot)) \,\,\,\,\delta J_{0}(x^{0}(\cdot), h_{1}(\cdot)) \\
\\
\delta J_{1}(x^{0}(\cdot), h(\cdot)) \,\,\,\, \delta
J_{1}(x^{0}(\cdot), h_{1}(\cdot))
\end{array}\right|=0
\end{equation}%
holds.

If the determinant (\ref{equ16}) is not equal to zero then the mapping
$(\gamma_{0}, \gamma_{1})\rightarrow (\Gamma_{0}(\gamma_{0},
\gamma_{1}), \ \Gamma_{1}(\gamma_{0}, \gamma_{1}))$ takes some
neighborhood of the point $(0, \ 0)$ to some neighborhood of the
point $(\Gamma_{0}(0, 0), \, \Gamma_{1}(0, 0))$. Therefore there
are such $\gamma_{0}$ and $\gamma_{1}$, and therefore an
admissible function
$x^{0}(\cdot)+\gamma_{0}h(\cdot)+\gamma_{1}h_{1}(\cdot),$ such
that
$$
\Gamma_{0}(\gamma_{0},
\gamma_{1})=J_{0}(x^{0}(\cdot)+\gamma_{0}h(\cdot)+\gamma_{1}h_{1}(\cdot))=J_{0}(x^{0}(\cdot))-\varepsilon,
$$
where $\varepsilon>0$, and $\Gamma_{1}(\gamma_{0},
\gamma_{1})=\Gamma_{1}(0, 0)=J_{1}(x^{0}(\cdot))=l_{1}.$

This means that the function $x^{0}$ does not provide a local
minimum in problem ($P_1$).

Therefore equalities (\ref{equ16}) are satisfied. From equalities (\ref{equ16}) we
have
$$
\delta J_{0}(x^{0}(\cdot), h(\cdot))\cdot \delta
J_{1}(x^{0}(\cdot), h_{1}(\cdot)) -\delta J_{0}(x^{0}(\cdot),
h_{1}(\cdot))\cdot \delta J_{1}(x^{0}(\cdot), h(\cdot))=0.
$$

Putting $\mu_{0}=1,$ $\mu_{1}=-\frac{\delta J_{0}(x^{0}(\cdot),
h_{1}(\cdot))}{\delta J_{1}(x^{0}(\cdot), h_{1}(\cdot))}$ we get
that
$$
\mu_{0}\delta J_{0}(x^{0}(\cdot), h(\cdot))+\mu_{1}\delta
J_{1}(x^{0}(\cdot), h(\cdot))\equiv 0
$$
for any $h(\cdot) \in PC_{0}^{\alpha}([t_{0}, t_{1}],
\mathbb{R}).$

Hence for any $h(\cdot) \in PC_{0}^{\alpha}([t_{0}, t_{1}],
\mathbb{R})$ we have
$$
\int\limits_{t_{0}}^{t_{1}}(t_{1}-t)^{\beta-1}[L_{x}(t, x^{0}(t),
(^{c}D_{t_{0}+}^{\alpha}x^{0})(t))h(t) +L_{y}(t, x^{0}(t),
(^{c}D_{t_{0}+}^{\alpha}x^{0})(t))(^{c}D_{t_{0}+}^{\alpha}h)(t)]dt=0.
$$

From here, by Theorem \ref{teo41}, we obtain that the Theorem \ref{teo51} is true.
\end{proof}

\begin{remark}\label{rem1} Note that an analogue of the necessary Euler-Lagrange condition for the simplest vector isoperimetric problem can be obtained similarly to the scalar case.
\end{remark}
\begin{eexample}\label{exam51} Consider the following fractional
isoperimetric problem
$$
J_{0}(x(\cdot))=\int\limits_{0}^{1}(1-t)^{\beta-1}((^{c}D_{0+}^{\alpha}x)(t))^{2}dt,
\ \ \ J_{1}(x(\cdot))=\int\limits_{0}^{1}(1-t)^{\beta-1}x(t)dt=0,
$$
$$
x(0)=0, \,\,\, x(1)=1, \,\,\, 0< \alpha \leq 1, \,\,\, \beta >0.
$$

The Lagrangian function has the form $
L=\mu_{0}(^{c}D_{0+}^{\alpha}x)^{2}+\mu_{1}x.$ In the case of
$\beta> \alpha >0$ from equation (\ref{equ14}) we get
\begin{equation}\label{equ17}
(1-t)^{1-\alpha}\frac{\mu_{1}}{\Gamma(\alpha)}
\int\limits_{t}^{1}(1-\tau)^{\beta-1}(\tau-t)^{\alpha-1}d\tau+
2\mu_{0}(^{c}D_{0+}^{\alpha}x)(t)=0.
\end{equation}%

If $\mu_{0}=0$, then from this equation it follows that
$\mu_{1}\frac{\Gamma(\beta)}{\Gamma(\alpha+\beta)}(1-t)^{\beta}=0.$
Thus, both Lagrange multipliers are equal to zero, which is impossible. Now let's put $\mu_{0}=\frac{1}{2}.$ Then
equation (\ref{equ17}) will take the form
$$
(^{c}D_{0+}^{\alpha}x)(t)=-\mu_{1}\frac{\Gamma(\beta)}{\Gamma(\alpha+\beta)}(1-t)^{\alpha},
\,\,\, t \in [0, \ 1].
$$

The solution to this equation has the form
$$
x(t)=x(0)-\frac{\mu_{1}\Gamma(\beta)}{\Gamma(\alpha)\Gamma(\alpha+\beta)}
\int\limits_{0}^{t}(t-\tau)^{\alpha-1}(1-\tau)^{\alpha}d\tau,
\,\,\, t \in [0, \ 1].
$$

Here, taking into account the boundary conditions, have
$$
\mu_{1}=-\frac{2\Gamma(\alpha+1)\Gamma(\alpha+\beta)}{\Gamma(\beta)}.
$$
And from the isoperimetric condition it follows that
$\frac{2\alpha
\Gamma(\beta)\Gamma(\alpha)}{(2\alpha+\beta)\Gamma(\alpha+\beta)}=0$,
which is impossible. Thus, in the case $\beta > \alpha>0$ the
isoperimetric problem has no solution.

Now consider the case $0< \beta\leq \alpha \leq 1.$ In this case,
equation (\ref{equ15}) has the form
\begin{equation}\label{equ18}
\frac{\mu_{1}(1-t)^{1-\beta}}{\Gamma(\alpha)}
\int\limits_{t}^{1}(1-\tau)^{\beta-1}(\tau-t)^{\alpha-1}d\tau +
2\mu_{0}(^{c}D_{0+}^{\alpha}x)(t)=\frac{k}{\Gamma(\alpha)}(1-t)^{\alpha-\beta}.
\end{equation}%

If $\mu_{0}=0,$ then from this equation we get $\mu_{1}=k
\frac{\Gamma(\alpha+\beta)}{\Gamma(\alpha)\Gamma(\beta)}(1-t)^{-\beta}.$
But this is impossible.

Now let's put $\mu_{0}=\frac{1}{2}.$ Then from equation (\ref{equ18}) it
follows that
$$
x(t)=x(0)+\frac{1}{\Gamma(\alpha)}\int\limits_{0}^{t}(t-\tau)^{\alpha-1}
\left[\frac{k}{\Gamma(\alpha)}(1-\tau)^{\alpha-\beta}-\mu_{1}(1-\tau)^{\alpha}\frac{\Gamma(\beta)}{\Gamma(\alpha+\beta)}\right]d\tau.
$$

Taking into account the boundary conditions and the isoperimetric
condition we have
$$
\mu_{1}=\frac{2\alpha(4\alpha^{2}-\beta^{2})\Gamma(\alpha)\Gamma(\alpha+\beta)}{\beta^{2}\Gamma(\beta)},
\,\, k=\frac{4(2\alpha-\beta)\Gamma^{2}(\alpha+1)}{\beta^{2}}
$$
and
$$
x(t)=\frac{4\alpha^{2}(2\alpha-\beta)}{\beta^{2}}\int\limits_{0}^{t}(t-\tau)^{\alpha-1}(1-\tau)^{\alpha-\beta}d\tau
-\frac{2\alpha(4\alpha^{2}-\beta^{2})}{\beta^{2}}\int\limits_{0}^{t}(t-\tau)^{\alpha-1}(1-\tau)^{\alpha}d\tau.
$$

Hence in the classical case, i.e. when $\alpha=\beta=1,$ it
follows that $x(t)=3t^{2}-2t$.
\end{eexample}

\section{ Weierstrass and Legendre necessary conditions for fractional calculus of variations
variations}\label{sec6}

In this section, we first derive a first-order necessary condition
for a strong extremum, then, by additionally requiring the
existence of a continuous derivative $L_{y}$, we obtain as a
corollary the Weierstrass necessary condition.

For the $0< \alpha \leq 1$ and $\beta >0$, we consider the following
simplest problem of fractional variational calculus in space
$PC^{\alpha}([t_{0}, t_{1}], \ \mathbb{R}^{n}):$
\begin{equation*}\label{equP}
J(x(\cdot))=\int\limits_{t_{0}}^{t_{1}}(t_{1}-t)^{\beta-1}L(t,
x(t), (^{c}D_{t_{0}+}^{\alpha}x)(t))dt \rightarrow extr, \,\,\,
x(t_{0})=x_{0}, \,\,\, x(t_{1})=x_{1}.\eqno{(P_2)} 
\end{equation*}%
Here the segment $[t_{0}, t_{1}]$ is assumed to be fixed and finite,
$t_{0}<t_{1}.$ $L=L(t, x, y)$ is a function $2n+1$ variables. The
extremum in problem ($P_2$) is considered among the space functions
$PC^{\alpha}([t_{0}, t_{1}], \mathbb{R}^{n})$ satisfying the
conditions at the ends or boundary conditions: $x(t_{0})=x_{0},$
$x(t_{1})=x_{1}.$ Such functions are called admissible.

\begin{definition}\label{def61} We will say that an admissible function
$x^{0}$ deliver a strong local minimum in problem ($P_2$), and write
$x^{0}\in strlocmin P_2,$ if there exists $\delta >0$ such that
$J(x(\cdot))\geq J(x^{0}(\cdot))$ for any admissible function $x$
for which
$$
\|x(\cdot)-x^{0}(\cdot)\|_{C([t_{0}, t_{1}],
\mathbb{R}^{n})}<\delta.
$$
\end{definition}

\begin{theorem}\label{teo61} Let the function $x^{0}\in
PC^{\alpha}([t_{0}, \,t_{1}], \, \mathbb{R}^{n})$ provide a strong
(weak) local minimum in problem ($P_2$) ($x^{0}\in strlocmin
P_2$($x^{0}\in wlocmin P_2$)), the functions $L, \, L_{x}$ are
continuous in some neighborhood of the graph $\{(t, x^{0}(t),
\,(^{c}D_{t_{0}+}^{\alpha}x^{0})(t))|t \in [t_{0}, t_{1}]\}$. Then
(Then there exists a number $\delta > 0$ such that) the inequality
$$
aL(\tau, x^{0}(\tau), (^{c}D_{t_{0}+}^{\alpha}x^{0})(\tau)+\xi)+
bL(\tau,
x^{0}(\tau),(^{c}D_{t_{0}+}^{\alpha}x^{0})(\tau)-\xi\frac{a}{b})
$$
\begin{equation}\label{equ19}
-(a+b)L(\tau, x^{0}(\tau),
(^{c}D_{t_{0}+}^{\alpha}x^{0})(\tau))\geq 0, \,\,\, \forall \tau
\in T, \,\, \forall a,\, b>0, \,\, \forall\xi \in \mathbb{R}^{n},
\end{equation}%
($\forall \tau \in T,$ $\forall a, \ b
>0,$ $\frac{a}{b}<1,$ $\forall \xi \in B_{\delta/2}(0)=\{\xi| \xi \in \mathbb{R}^{n}, \ \|\xi\|\leq \frac{\delta}{2}\}$ ) is satisfied on $x^{0}$.
\end{theorem}
\begin{proof} Let the point $\tau \in [t_{0}, \ t_{1}]\cap T$,
the numerical parameters $a, \ b \in (0, \ +\infty)$ and the
vector $\xi \in \mathbb{R}^{n}$ be fixed. We will choose
$\varepsilon_{1}>0$ so that $[\tau, \
\tau+(a+b)\varepsilon_{1})\subset T$. Let's introduce the
functions
$$
\psi[l_{1}, \ l_{2}](t,
\varepsilon)=(t-(\tau+l_{1}\varepsilon))^{\alpha}-(t-(\tau+l_{2}\varepsilon))^{\alpha},
\,\,t \geq \tau+l_{2}\varepsilon, \, l_{2}> l_{1},
$$
$$
 l_{i}\in \{0, \ a, \ a+b\}, \,\, i=1, 2, \,\, \varepsilon \in
(0, \varepsilon_{1}],
$$
$$
\varphi(t, \varepsilon)=\psi[0, a](t,
\varepsilon)-\frac{a}{b}\psi[a, a+b](t, \varepsilon),\,\,\, t \geq
\tau+(a+b)\varepsilon,
$$
and
$$
k(\varepsilon)=\frac{\varphi(t_{1}, \varepsilon)}{\psi[0,
a+b](t_{1}, \varepsilon)}, \,\,\, \varepsilon \in (0, \
\varepsilon_{1}].
$$

Using the rule of G.L'Hopital, we obtain that
\begin{equation}\label{equ20}
\lim\limits_{\varepsilon\rightarrow 0}k(\varepsilon)=0.
\end{equation}%

By direct calculation we can verify that
\begin{equation}\label{equ21}
\varphi(t,
\varepsilon)=\frac{\varepsilon^{2}}{2}\alpha(\alpha-1)a(a+b)(t-\tau)^{\alpha-2}+
o(\varepsilon^{2},t), \,\,\, t \geq \tau+(a+b)\varepsilon,
\end{equation}%
\begin{equation}\label{equ22}
\psi[0, a+b](t, \varepsilon)=\varepsilon \alpha
(a+b)(t-\tau)^{\alpha-1}+ o(\varepsilon,t), \,\,\, t \geq
\tau+(a+b)\varepsilon.
\end{equation}%

Next, assume that $\varepsilon_{2}\in (0, \varepsilon_{1}]$ is a
number such that for any $\varepsilon \in (0, \varepsilon_{2}]$ the
inequality $|k(\varepsilon)|<a_{0}$ holds, where $a_{0}=\min\{1, \
\frac{a}{b}\}.$

For any $\varepsilon \in (0, \varepsilon_{2}]$ we define a special
variation $h(t, \varepsilon)=h(t, \tau, a, b, \xi, \varepsilon)$ of
the function $x^{0}$ as follows:
\begin{equation*}
    h(t,\varepsilon)=
    \begin{cases}
        0,&t\in [t_{0},\ \tau],\  \\
        \frac{\xi(1-k(\varepsilon))}{\Gamma (\alpha +1)}(t-\tau)^{\alpha},&t\in[\tau, \tau+a \varepsilon], \\
        \frac{\xi(1-k(\varepsilon))}{\Gamma (\alpha +1)}\psi[0, a](t,\varepsilon)-\frac{\xi(\frac{a}{b}+k(\varepsilon))}{\Gamma (\alpha+1)}(t-(\tau+a \varepsilon))^{\alpha},&  t \in [\tau+a\varepsilon,\tau+(a+b)\varepsilon],\ \\
        \frac{\xi}{\Gamma (\alpha +1)}[\varphi(t, \varepsilon)-k(\varepsilon)\psi[0,a+b](t, \varepsilon)], & t \in [\tau+(a+b)\varepsilon, t_{1}].
    \end{cases}
\end{equation*}
It follows that for any $\varepsilon \in (0, \varepsilon_{2}]$ the
function $h(\cdot, \varepsilon)\in C([t_{0}, t_{1}],
\mathbb{R}^{n})$ and $h(t_{0}, \varepsilon)=h(t_{1},
\varepsilon)=0$ . Using the inequality
$$
\sigma_{2}^{\alpha}-\sigma_{1}^{\alpha}\leq
(\sigma_{2}-\sigma_{1})^{\alpha}, \,\,\, 0< \sigma_{1}\leq
\sigma_{2}, \,\, 0<\alpha\leq 1,
$$
we can show that for any  $\varepsilon \in (0, \varepsilon_{2}]$ the
inequality
\begin{equation}\label{equ23}
\|h(\cdot)\|_{C([t_{0}, t_{1}], R^{n})}\leq M
\varepsilon^{\alpha}, 
\end{equation}%
holds, where
$M=\frac{\|\xi\|}{\Gamma(\alpha+1)}\left[2a^{\alpha}+\left(\frac{a}{b}+1
\right)b^{\alpha}\right].$

The fractional Caputo derivative of the  $\alpha$-th order of the
function $h(t, \varepsilon) $ has the form
\begin{equation*}
(^{c}D_{t_{0}+}^{\alpha}h)(t,\varepsilon)=
    \begin{cases}
        \xi(1-k(\varepsilon)),& t\in [\tau,\ \tau+a\varepsilon],\  \\
        -\xi(\frac{a}{b}+k(\varepsilon)), & t\in[\tau+a\varepsilon, \tau+(a+b) \varepsilon],\ \\
        0, & t \in [t_{0}, \tau)\cup (\tau+(a+b)\varepsilon, t_{1}].
    \end{cases}
\end{equation*}
Hence the function $h(\cdot, \varepsilon)\in
PC_{0}^{\alpha}([t_{0}, t_{1}], \mathbb{R}^{n}),$ $\varepsilon \in
(0, \varepsilon_{2}].$

Now let's put
$$
x(t, \varepsilon)=x^{0}(t)+h(t,\varepsilon), \,\,\, t \in [t_{0},
t_{1}], \,\,\, \varepsilon \in (0, \varepsilon_{2}].
$$

It is obvious that the function $x(t, \varepsilon)$ is admissible
in problem ($P_2$), i.e. $x(\cdot, \varepsilon)\in PC^{\alpha}([t_{0},
t_{1}], \mathbb{R}^{n})$ and $x(t_{i}, \varepsilon)=x_{i},$ $i=0,
\ 1$, $\forall \varepsilon \in (0, \varepsilon_{2}]$. From
inequality (\ref{equ23}) it follows that $x(t, \varepsilon)\rightarrow
x_{0}(t)$ in the metric of space $C([t_{0}, t_{1}],
\mathbb{R}^{n})$ with $\varepsilon \rightarrow 0$. Therefore
\begin{equation}\label{equ24}
\lim\limits_{\varepsilon\rightarrow 0}x(t, \varepsilon)=x_{0}(t),
\,\,\, t \in [t_{0}, t_{1}].
\end{equation}%

Using expansions (\ref{equ21}) and (\ref{equ22}), for sufficiently small
$\varepsilon \in (0, \varepsilon_{2}]$ for the function $h(t,
\varepsilon)$ on the segment $[\tau+(a+b)\varepsilon, t_{1}]$ we
obtain the following estimate
$$
\|h(t, \varepsilon)\|\leq
\frac{\|\xi\|\varepsilon^{2}}{\Gamma(\alpha+1)}
\left[\frac{\alpha(1-\alpha)a(a+b)}{2}(t-\tau)^{\alpha-2}+
\frac{|o(\varepsilon^{2},t)|}{\varepsilon^{2}}\right.
$$
$$
\left.+\left(\frac{\alpha(1-\alpha)a(a+b)}{2}(t_{1}-\tau)^{\alpha-2}+\frac{|o(\varepsilon^{2},
t_{1})|}{\varepsilon^{2}}\right)\frac{\alpha(a+b)(t-\tau)^{\alpha-1}+\frac{|o(\varepsilon,
t)|}{\varepsilon}}{\alpha(a+b)(t_{1}-\tau)^{\alpha-1}+\frac{|o(\varepsilon,
t_{1})|}{\varepsilon}}\right]
$$
\begin{equation}\label{equ25}
\leq
\frac{\|\xi\|\varepsilon^{2}}{\Gamma(\alpha+1)}\left[\frac{\alpha(1-\alpha)a(a+b)(t-\tau)^{\alpha}}{(t-\tau)^{2}}+
\frac{2\alpha(1-\alpha)a(a+b)}{(t_{1}-\tau)(t-\tau)}(t-\tau)^{\alpha}\right]\leq
k_{1}(t-\tau)^{\alpha-2}\varepsilon^{2}, 
\end{equation}%
where $k_{1}=\frac{3\|\xi\|(1-\alpha)a(a+b)}{\Gamma(\alpha)}.$

Now let's calculate the increment of the functional $J(x(\cdot))$:
$$
\Delta J(x(\cdot))=J(x(\cdot,
\varepsilon))-J(x^{0}(\cdot))=\int\limits_{t_{0}}^{t_{1}}(t_{1}-t)^{\beta-1}[L(t,
x(t, \varepsilon), (^{c}D_{t_{0}+}^{\alpha}x)(t, \varepsilon))
$$
$$
-L(t, x^{0}(t),
(^{c}D_{t_{0}+}^{\alpha}x^{0})(t))]dt=\int\limits_{\tau}^{\tau+a\varepsilon}(t_{1}-t)^{\beta-1}
[L(t, x(t, \varepsilon),
(^{c}D_{t_{0}+}^{\alpha}x^{0})(t)+\xi(1-k(\varepsilon)))
$$
$$
-L(t, x^{0}(t),
(^{c}D_{t_{0}+}^{\alpha}x^{0})(t))]dt+\int\limits_{\tau+a\varepsilon}^{\tau+(a+b)\varepsilon}(t_{1}-t)^{\beta-1}
[L(t, x(t, \varepsilon),
(^{c}D_{t_{0}+}^{\alpha}x^{0})(t)-\xi(\frac{a}{b}+k(\varepsilon)))
$$
$$
-L(t, x^{0}(t), (^{c}D_{t_{0}+}^{\alpha}x^{0})(t))]dt
+\int\limits_{\tau+(a+b)\varepsilon}^{t_{1}}(t_{1}-t)^{\beta-1}
[L(t, x(t,\varepsilon), (^{c}D_{t_{0}+}^{\alpha}x^{0})(t))
$$
\begin{equation}\label{equ26}
- L(t, x^{0}(t),
(^{c}D_{t_{0}+}^{\alpha}x^{0})(t))]dt=:J_{1}(\varepsilon)+J_{2}(\varepsilon)+J_{3}(\varepsilon).
\end{equation}%

Using the mean value theorem for definite integrals, as well as
equalities (\ref{equ20}) and (\ref{equ24}), we have
\begin{equation}\label{equ27}
\lim \limits_{\varepsilon\rightarrow
0}\frac{J_{1}(\varepsilon)}{\varepsilon}=a(t_{1}-\tau)^{\beta-1}
[L(\tau, x^{0}(\tau), (^{c}D_{t_{0}+}^{\alpha}x^{0})(\tau)+\xi)
-L(\tau, x^{0}(\tau), (^{c}D_{t_{0}+}^{\alpha}x^{0})(\tau))],
\end{equation}%
\begin{equation}\label{equ28}
\lim \limits_{\varepsilon\rightarrow
0}\frac{J_{2}(\varepsilon)}{\varepsilon}=b(t_{1}-\tau)^{\beta-1}
[L(\tau, x^{0}(\tau),
(^{c}D_{t_{0}+}^{\alpha}x^{0})(\tau)-\xi\frac{a}{b}) -L(\tau,
x^{0}(\tau), (^{c}D_{t_{0}+}^{\alpha}x^{0})(\tau))]. 
\end{equation}%

Now we estimate $J_{3}(\varepsilon)$. First consider the case
$0<\beta<1$. Taking into account the continuous differentiability
of the integrand with respect to $x$ in some  neighborhood of the
graph $\{(t, x_{0}(t), (^{c}D_{t_{0}+}^{\alpha}x_{0})(t))| t \in
[t_{0}, t_{1}]\}$ and inequality (\ref{equ25}), we have
$$
|J_{3}(\varepsilon)|=|\int\limits_{\tau+(a+b)\varepsilon}^{t_{1}}(t_{1}-t)^{\beta-1}L_{x}(t,
x^{0}(t)+\mu
h(t,\varepsilon),(^{c}D_{t_{0}+}^{\alpha}x^{0})(t))h(t,
\varepsilon)dt|
$$
$$
\leq
k_{3}\varepsilon^{2}\int\limits_{\tau+(a+b)\varepsilon}^{t_{1}}(t_{1}-t)^{\beta-1}
(t-\tau)^{\alpha-2}dt
$$
$$
=k_{3}\varepsilon^{2}\left[\int\limits_{\tau+(a+b)\varepsilon}^{\eta}(t_{1}-t)^{\beta-1}(t-\tau)^{\alpha-2}dt
+\int\limits_{\eta}^{t_{1}}(t_{1}-t)^{\beta-1}(t-\tau)^{\alpha-2}dt
\right]
$$
$$
\leq
k_{3}\varepsilon^{2}\left[\frac{(t_{1}-\eta)^{\beta-1}}{1-\alpha}(((a+b)\varepsilon)^{\alpha-1}
-(\eta-\tau)^{\alpha-1})+\frac{(t_{1}-\eta)^{\beta}}{\beta}
(\eta-\tau)^{\alpha-2}\right]
$$
\begin{equation}\label{equ29}
\leq
\varepsilon^{1+\alpha}\left[\frac{k_{3}(a+b)^{\alpha-1}}{1-\alpha}(t_{1}-\eta)^{\beta-1}+\frac{k_{3}\varepsilon^{1-\alpha}}{\beta}(t_{1}-\eta)^{\beta}
(\eta-\tau)^{\alpha-2}\right],
\end{equation}%
where
$$
k_{2}={{\rm
esssup}}_{t \in [t_{0}, t_{1}]}|L_{x}(t, x^{0}(t)+\mu h(t,
\varepsilon), (^{c}D_{t_{0}+}^{\alpha}x^{0})(t))|,
$$
$$
k_{3}=k_{2}\cdot k_{1}, \,\,\, \tau+(a+b)\varepsilon< \eta< t_{1},
\,\,\, 0<\mu<1.
$$

Now consider the case $\beta \geq 1$. Then we have
$$
|J_{3}(\varepsilon)|\leq
k_{3}\varepsilon^{2}(t_{1}-\tau)^{\beta-1}\int\limits_{\tau+(a+b)\varepsilon}^{t_{1}}(t-\tau)^{\alpha-2}dt
$$
\begin{equation}\label{equ30}
=\frac{k_{3}\varepsilon^{2}(t_{1}-\tau)^{\beta-1}}{1-\alpha}
[((a+b)\varepsilon)^{\alpha-1}-(t_{1}-\tau)^{\alpha-1})]\leq
\frac{k_{3}(a+b)^{\alpha-1}}{1-\alpha}(t_{1}-\tau)^{\beta-1}
\varepsilon^{1+\alpha}. 
\end{equation}%

From estimates (\ref{equ29}) and (\ref{equ30}) it follows that
\begin{equation}\label{equ31}
\lim \limits_{\varepsilon\rightarrow
0}\frac{1}{\varepsilon}J_{3}(\varepsilon)=0. 
\end{equation}%

Since $x^{0}\in strlocmin P_2,$ then the function of one variable
$\Phi(\varepsilon)=J(x(\cdot, \varepsilon))$ has a minimum at the
point $\varepsilon=0,$ i.e. $\Phi(\varepsilon)\geq \Phi(0).$ It
follows that if the derivative on the right at zero $\Phi'(+0)$
exists, then $\Phi'(+0)\geq 0.$

Then taking into account (\ref{equ27}), (\ref{equ28}) and (\ref{equ31}) in (\ref{equ26}) we get
$$
\Phi'(+0)=\lim \limits_{\varepsilon\rightarrow
+0}\frac{\Phi(\varepsilon)-\Phi(0)}{\varepsilon} =\lim
\limits_{\varepsilon\rightarrow 0}\frac{1}{\varepsilon}\Delta
J(x(\cdot))=(t_{1}-\tau)^{\beta-1}\{a[L(\tau, x^{0}(\tau),
(^{c}D_{t_{0}+}^{\alpha}x^{0})(\tau)+\xi)
$$
$$
-L(\tau, x^{0}(\tau),
(^{c}D_{t_{0}+}^{\alpha}x^{0})(\tau))]+b[L(\tau, x^{0}(\tau),
(^{c}D_{t_{0}+}^{\alpha}x^{0})(\tau)-\xi\frac{a}{b})
$$
$$
-L(\tau, x^{0}(\tau), (^{c}D_{t_{0}+}^{\alpha}x^{0})(\tau))]\}\geq
0.
$$
It follows that inequality (\ref{equ19}) is true. The case of a local
minimum is proved similarly.  
\end{proof}

To show the geometric meaning of condition (\ref{equ19}), dividing both
parts of condition (\ref{equ19}) by $a+b$, we get

\[L\left(t,\ x^0\left(t\right),\left({}^c{D^{\alpha }_{t_0+}}x^0\right)\left(t\right)\right)
=L\left(t,\ x^0\left(t\right),\
\frac{a}{a+b}\left(\left({}^c{D^{\alpha
}_{t_0+}}x^0\right)\left(t\right)+\xi \right)+\right.\]
\[\left.+\ \frac{b}{a+b}\left(\left({}^c{D^{\alpha }_{t_0+}}x^0\right)\left(t\right)-\xi \frac{a}{b}\right)\right)\le
\frac{a}{a+b}L\left(t,x^0\left(t\right),\left({}^c{D^{\alpha
}_{t_0+}}x^0\right)\left(t\right)+\xi \right)+\ \]
\[+\frac{b}{a+b}L\left(t,x^0\left(t\right),\left({}^c{D^{\alpha }_{t_0+}}x^0\right)\left(t\right)
-\xi \frac{a}{b}\right),\ \ \ \ \ \ t\in T,\ \ \ a,\ b>0,\ \ \ \xi
\in \mathbb{R}^n.\]

From this inequality in the case $\ n=1 \ $ it follows that
condition (\ref{equ19}) has a clear geometric meaning on the extremal $ x^0
\ $ : for a fixed $ \ t \in T, \ $ the point (\
$\left({}^c{D^{\alpha }_{t_0+}}x^0\right)\left(t\right), \
L\left(t,x^0\left(t\right), \right. \ $ $\left.
\left({}^c{D^{\alpha }_{t_0+}}x^0\right)\left(t\right)\right)$
lies no higher than any chord with ends on different sides of
${}^c{D^{\alpha }_{t_0+}}x^0$ on the graph of the function
$L=L\left({}^c{D^{\alpha }_{t_0+}}x\right)=L(t,\
x^0\left(t\right),\ {}^c{D^{\alpha }_{t_0+}}x)$ (as a function of
${}^c{D^{\alpha }_{t_0+}}x$).

Note that condition (\ref{equ19}) is always satisfied if for each $t\in T$
and $x\in \mathbb{R}^n$ the integrand $L(t,\ x,\ {}^c{D^{\alpha
}_{t_0+}}x)$ is a convex function with respect to ${}^c{D^{\alpha
}_{t_0+}}x.$

\begin{corollary}\label{cor61}\textbf{(The Weierstrass condition)} Let the
function $x^{0}\in PC^{\alpha}([t_{0}, t_{1}], \mathbb{R}^{n})$
provide a strong (weak) local minimum in problem ($P_2$) ($x^{0} \in
strlocmin P_2$ $(x^{0} \in wlocmin P_2)$), the functions $L, \ L_{x},
\ L_{y}$ are continuous in some neighborhood of the graph $\{(t,
x^{0}(t),$ $(^{c}D_{t_{0}+}^{\alpha}x^{0})(t))|t\in [t_{0},
t_{1}]\}$. Then (Then there exists a number $\delta >0$ such that)
the inequality
\begin{equation}\label{equ32}
E(\tau, x^{0}(\tau),(^{c}D_{t_{0}+}^{\alpha}x^{0})(\tau),
(^{c}D_{t_{0}+}^{\alpha}x^{0})(\tau)+\xi)\geq 0, \,\,\,\forall
\tau \in T, \,\, \forall \xi \in \mathbb{R}^{n} 
\end{equation}%
$$
(\forall \tau \in T, \,\,\, \forall \xi \in B_{\delta/2}(0)=\{\xi|
\xi \in \mathbb{R}^{n}, \ \|\xi\|\leq \frac{\delta}{2}\}),
$$
is satisfied on $x^{0},$ where $E(t, x, y, z)=L(t, x, z)-L(t, x,
y)-\langle L_{y}(t, x, y), z-y \rangle$ is a Weierstrass function.
\end{corollary}
\begin{proof} By Taylor formula we have
$$
L(\tau, x^{0}(\tau),
(^{c}D_{t_{0}+}^{\alpha}x^{0})(\tau)-\frac{a}{b}\xi)= L(\tau,
x^{0}(\tau), (^{c}D_{t_{0}+}^{\alpha}x^{0})(\tau))
$$
$$
-\frac{a}{b}\langle L_{y}(\tau, x^{0}(\tau),
(^{c}D_{t_{0}+}^{\alpha}x^{0})(\tau)), \xi\rangle+o(a),
$$
where $\lim \limits_{a\rightarrow +0}\frac{o(a)}{a}=0,$ $\tau \in
T,$ $a, \ b>0,$ and $\xi \in \mathbb{R}^{n}.$

Taking into account this expansion in (\ref{equ19}), we get
$$
a E(\tau, x^{0}(\tau),(^{c}D_{t_{0}+}^{\alpha}x^{0})(\tau),
(^{c}D_{t_{0}+}^{\alpha}x^{0})(\tau)+\xi)+b \ o(a)\geq 0.
$$

Dividing both sides of this inequality by $a$ and passing to the
limit at $a \rightarrow +0$, we get (\ref{equ32}). The case of a local
minimum is  proved similarly. 
\end{proof}

\begin{remark}\label{rem61} Under the conditions of Corollary \ref{cor61},
inequalities (\ref{equ19}) and (\ref{equ32}) are equivalent.

To prove this fact, it is sufficient to show that if inequality
(\ref{equ32}) is satisfied, the inequality (\ref{equ19}) is also satisfied. Let
inequality (\ref{equ32}) be satisfied. Then for all $\tau \in T,$ $a,b>0$
and $\xi \in \mathbb{R}^n$ the following inequalities hold:

\[L\left(\tau ,x^0\left(\tau \right),\left({}^c{D^{\alpha }_{t_0+}}x^0\right)
\left(\tau \right)+\xi \right)-L\left(\tau ,x^0\left(\tau
\right),\left(^{c}{D^{\alpha }_{t_0+}}x^0\right)\left(\tau
\right)\right) \]

\[\ge \left\langle L_y\left(\tau ,x^0\left(\tau \right),
\left({}^c{D^{\alpha }_{t_0+}}x^0\right)\left(\tau
\right)\right),\xi \right\rangle ,\]

\[L\left(\tau ,x^0\left(\tau \right),\left({}^c{D^{\alpha }_{t_0+}}
x^0\right)\left(\tau \right)-\frac{a}{b}\xi \right)-L\left(\tau
,x^0\left(\tau \right), \left({}^c{D^{\alpha
}_{t_0+}}x^0\right)\left(\tau \right)\right) \]
\[\ge -\frac{a}{b}\left\langle L_y\left(\tau ,x^0\left(\tau \right),
\left({}^c{D^{\alpha }_{t_0+}}x^0\right)\left(\tau
\right)\right),\xi \right\rangle .\] We multiply both parts of the
first inequality by $a>0$, the second inequality by $b>0$ and
adding the resulting inequalities, we obtain inequality (\ref{equ19}). This
completes the proof. 
\end{remark}
To illustrate inequality (\ref{equ19}), consider the following example.\\
\textbf{Exsample 1} Let us consider the problem
\begin{equation}\label{equ33}
 \int \limits_{0}^{1}(1-t)^{\beta-1}[((^{c}D_{0+}^{\alpha}x)(t))^{5}+
|(^{c}D_{0+}^{\alpha}x)(t)|]dt \rightarrow \min, \,\,\,
x(0)=x(1)=0,
\end{equation}%
where $\beta > 0, \, 0<\alpha\leq 1,$ $x\in PC^{\alpha}([0, 1], \
\mathbb{R})$, $L(t, x,
^{c}D_{0+}^{\alpha}x)=(^{c}D_{0+}^{\alpha}x)^{5}+|^{c}D_{0+}^{\alpha}x|.$

Along the admissible function $x^{0}(t)=0,$ $t \in [0, \ 1]$ we
have
$$
L(t, x^{0}(t), (^{c}D_{0+}^{\alpha}x^{0})(\tau))=0, \,\,\, L(t,
x^{0}(t), (^{c}D_{0+}^{\alpha}x^{0})(t)+\xi)
$$
$$
=\xi^{5}+|\xi|, \,\,\, L(t, x^{0}(t),
(^{c}D_{0+}^{\alpha}x^{0})(t)-\xi\frac{a}{b})=-
\xi^{5}\frac{a^{5}}{b^{5}}+|\xi|\frac{a}{b}, \,\, \xi \in
\mathbb{R}.
$$

Then inequality (\ref{equ19}) will take the form
$$
a(\xi^{5}+|\xi|)+b(-\xi^{5}\frac{a^{5}}{b^{5}}+|\xi|\frac{a}{b})
=a\xi^{5}(1-\frac{a^{4}}{b^{4}})+2a|\xi|\geq 0.
$$

It is obvious that this inequality does not hold, for example, for
$a=2, \, b=1$ and  $\xi=1$, therefore in problem (\ref{equ33}) the
admissible function $x^{0}(t)=0,$ $t \in [0, \ 1]$ is not a strong
local minimum.

If $\frac{a}{b}<1$ and $|\xi|\leq \frac{1}{2},$ then inequality
(\ref{equ19}) takes the form
$$
a \xi^{4}\left[\left(1-\frac{a^{4}}{b^{4}}\right)\xi+2 \right]\geq
0
$$
and therefore, the function $x^{0}=0$ can be a local minimum in
problem (\ref{equ33}). On the other hand, for
$|(^{c}D_{0+}^{\alpha}x)(t)|\leq 1,$ it follows from the
inequality
$$
(^{c}D_{0+}^{\alpha}x)^{5}(t)+|(^{c}D_{0+}^{\alpha}x)(t)|\geq
(^{c}D_{0+}^{\alpha}x)^{4}(t)((^{c}D_{0+}^{\alpha}x)(t)+1)\geq 0,
$$
that the function $x^{0}=0$ is indeed a local minimum in problem
(\ref{equ33}).

\textbf{Example 2} Let us consider the problem
\begin{equation}\label{equ34}
\int
\limits_{0}^{1}(1-t)^{\beta-1}(^{c}D_{0+}^{\alpha}x)^{3}(t)dt\rightarrow
\min, \,\,\, x(0)=0, \,\, x(1)=1, 
\end{equation}%
where $\beta > 0, \, 0<\alpha\leq 1,$ $x \in PC^{\alpha}([0, 1], \
\mathbb{R})$, $L(t, x,
^{c}D_{0+}^{\alpha}x)=(^{c}D_{0+}^{\alpha}x)^{3}.$

To solve the problem we use Theorem \ref{teo41}. Let us first consider the
case $0 < \alpha < \beta$. In this case, by condition (\ref{equ5}), we have
$$
(^{c}D_{0+}^{\alpha}x)^{2}(t)=0, \,\,\, t \in [0, \ 1].
$$
Hence we get $x(t)=x(0)=0,$ $t \in [0, \ 1]$. Therefore, problem
(\ref{equ34}) does not have solution.

Now consider the case $0 < \beta \leq \alpha \leq 1.$ In this case
equation (\ref{equ6}) takes the form
$$
3(^{c}D_{0+}^{\alpha}x)^{2}(t)=\frac{k}{\Gamma(\alpha)}(1-t)^{\alpha-\beta},
\,\,\, k \in \mathbb{R}, \,\,\, t \in [0, \ 1].
$$

The solution of this equation, satisfying the conditions $x(0)=0,$
$x(1)=1,$ has the form
\begin{equation}\label{equ35}
x^{0}(t)=\frac{3\alpha-\beta}{2}\int\limits_{0}^{t}(t-\tau)^{\alpha-1}(1-\tau)^{\frac{\alpha-\beta}{2}}d\tau.
\end{equation}%

From here we have
$(^{c}D_{0+}^{\alpha}x^{0})(t)=\frac{(3\alpha-\beta)\Gamma(\alpha)}{2}(1-t)^{\frac{\alpha-\beta}{2}}.$
In the case of $0< \alpha=\beta \leq 1$ we get
$x^{0}(t)=t^{\alpha}$ and
$(^{c}D_{0+}^{\alpha}x^{0})(t)=\Gamma(\alpha+1).$ We will show
that for $0< \beta=\alpha \leq 1$ the extremal
$x_{0}(t)=t^{\alpha}$ provides a weak local minimum for problem
(\ref{equ34}). Indeed, let $h \in PC_{0}^{\alpha}([0, \ 1], \mathbb{R}).$
Then
$$
J(x^{0}(\cdot)+h(\cdot))-J(x^{0}(\cdot))=\int\limits_{0}^{1}(1-t)^{\alpha-1}
[3\Gamma^{2}(\alpha+1)(^{c}D_{0+}^{\alpha}h)(t)
$$
$$
+3\Gamma(\alpha+1)(^{c}D_{0+}^{\alpha}h)^{2}(t)+(^{c}D_{0+}^{\alpha}h)^{3}(t)]dt
=\int\limits_{0}^{1}(1-t)^{\alpha-1}(^{c}D_{0+}^{\alpha}h)^{2}(t)
[3\Gamma(\alpha+1)+(^{c}D_{0+}^{\alpha}h)(t)]dt.
$$

From this it is clear that if $\|h(\cdot)\|_{1} <
3\Gamma(\alpha+1),$ then
$3\Gamma(\alpha+1)+(^{c}D_{0+}^{\alpha}h)(t)> 0$ and,
consequently,
$$
J(x^{0}(\cdot)+h(\cdot))\geq J(x^{0}(\cdot)), \,\,\,\,
\texttt{i.e.}\,\, x^{0}\in locmin.
$$

Let us show that (\ref{equ35}) does not provide a strong extremum for
problem (\ref{equ34}). It is obvious that the Weierstrass function has the
form $ E=\xi^{2}(\xi+3(^{c}D_{0+}^{\alpha}x^{0})(t)).$ At
$\xi\neq$
$-\frac{3(3\alpha-\beta)}{2}\Gamma(\alpha)(1-t)^{\frac{\alpha-\beta}{2}}$
the Weierstrass function can take both positive and negative
values. Therefore the extremal (\ref{equ35}) does not give a strong
extremum to the functional.

\begin{corollary}\label{cor62} \textbf{(The Legendre condition)} In addition to
the conditions of Corollary \ref{cor61}, let the integrant $L(t,x,y)\ $ be
twice differentiable with respect to $\, {}^c{D^{\alpha }_{t_0+}}x
\, $ at the points of the set $\ \Gamma
=\left\{(t,x^0\left(t\right),\left({}^c{D^{\alpha
}_{t_0+}}x^0\right)\right.$ $\left.(t))\left|t\in T\right.\
\right\}.$ Then, in order for the admissible function $x^0$ to
provide a weak local minimum in problem ($P_2$), it is necessary that
for any point $t\in T$ and any vector $\mu \in \mathbb{R}^n$ the
following inequality holds
\begin{equation}\label{equ36}
\left\langle L_{yy}(t)\mu ,\mu \right\rangle \ge 0, 
\end{equation}%
where
$L_{yy}\left(t\right)=L_{yy}\left(t,x^0(t),\left({}^c{D^{\alpha
}_{t_0+}}x^0\right)\left(t\right)\right).$
\end{corollary}
\begin{proof} Let the admissible function $x^0$ be a weak local
minimum in problem ($P_2$). Then it is obvious that under the
conditions of Corollary \ref{cor61}, the inequality (\ref{equ32}) is satisfied. Now
we fix the point $\tau \in T$ and the vector $\mu \in
\mathbb{R}^n$. Then we can choose the number ${\varepsilon }_1>0$
so that for all $\varepsilon \in (-{\varepsilon }_1,\ {\varepsilon
}_1)$ the inclusion $\varepsilon \mu \in B_{\frac{\delta }{2}}(0)$
is satisfied. By Taylor's formula we obtain

\[L\left(\tau ,x^0\left(\tau \right),\left({}^c{D^{\alpha }_{t_0+}}x^0\right)
\left(\tau \right)+\xi \mu \right)=L\left(\tau ,x^0\left(\tau
\right), \left({}^c{D^{\alpha }_{t_0+}}x^0\right)\left(\tau
\right)\right)\]
\[+\varepsilon \left\langle L_y\left(\tau \right),\mu \right\rangle
+\frac{{\varepsilon }^2}{2}\left\langle L_{yy}\left(\tau \right)\mu
,\mu \right\rangle +o\left({\varepsilon }^2\right).\]

 Hence, taking into account the Weierstrass function and inequality (\ref{equ32}), we have
\[E\left(\tau ,x^0\left(\tau \right),\left({}^c{D^{\alpha }_{t_0+}}x^0\right)
\left(\tau \right),\ \left({}^c{D^{\alpha
}_{t_0+}}x^0\right)\left(\tau \right) +\varepsilon \mu
\right)=\frac{{\varepsilon }^2}{2}\left\langle L_{yy}\left(\tau
\right)\mu,\mu \right\rangle +o\left({\varepsilon }^2\right)\ge
0,\] where $\lim \limits_{\varepsilon \to 0} \frac{o({\varepsilon
}^2)}{{\varepsilon }^2}=0.$

Dividing both parts of the last inequality by ${\varepsilon }^2$
and passing to the limit at $\varepsilon \to 0$, we obtain
inequality (\ref{equ36}).
\end{proof}

\begin{corollary}\label{cor63} \textbf{(The second Weierstrass-Erdman condition).}
Let $0<\alpha \le 1,$ $\beta
>0$ and the vector function $x^0$ provides a strong local minimum for problem ($P_2$), then at any corner point $t\in A$ of the
function $x^0$ the equality

\[L\left(t,\ x^0\left(t\right),\left({}^c{D^{\alpha
}_{t_0+}}x^0\right)\left(t-0\right)\right)-\left\langle
L_y\left(t,x^0\left(t\right),\left({}^c{D^{\alpha
}_{t_0+}}x^0\right)\left(t-0\right)\right),\left({}^c{D^{\alpha
}_{t_0+}}x^0\right)\left(t-0\right)\right\rangle \]
\begin{equation}\label{equ37}
=L\left(t,\ x^0\left(t\right),\left({}^c{D^{\alpha
}_{t_0+}}x^0\right)\left(t+0\right)\right)-\left\langle
L_y\left(t,x^0\left(t\right),\left({}^c{D^{\alpha
}_{t_0+}}x^0\right)\left(t+0\right)\right),\left({}^c{D^{\alpha
}_{t_0+}}x^0\right)\left(t+0\right)\right\rangle  
\end{equation}%
holds.
\end{corollary}
\begin{proof} Since the vector function $x^0$ provides a strong
local extremum for problem ($P_2$), this extremum is also a weak local
extremum. Then, according to Corollary \ref{cor41}, the equality
\begin{equation}\label{equ38}
L_y\left(t,x^0\left(t\right),\left({}^c{D^{\alpha
}_{t_0+}}x^0\right)\left(t-0\right)\right)=L_y\left(t,x^0\left(t\right),\left({}^c{D^{\alpha
}_{t_0+}}x^0\right)\left(t+0\right)\right),\ \ t\in A, 
\end{equation}%
is satisfied.

To obtain equality (\ref{equ37}), we set

\[{\mathcal Z}\left(t,y\right)=L\left(t,\ x^0\left(t\right),\ y\right)-\left\langle
L_y\left(t,x^0\left(t\right),\left({}^c{D^{\alpha
}_{t_0+}}x^0\right)\left(t\right)\right),y\right\rangle .\]

From Corollary \ref{cor61} it follows that the following inequalities are
satisfied
\[E(t, x^{0}(t), (^{c}D_{t_{0}+}^{\alpha}x^{0})(t-0),
(^{c}D_{t_{0}+}^{\alpha}x^{0})(t+0))\]
\[={\mathcal Z}\left(t,\ \left({}^c{D^{\alpha }_{t_0+}}x^0\right)
\left(t+0\right)\right)-{\mathcal Z}\left(t,\ \left({}^c{D^{\alpha
}_{t_0+}}x^0\right)\left(t-0\right)\right)\ge 0,\ \ t\in A,\] and
\[E\left(t,\ x^0\left(t\right),\left({}^c{D^{\alpha }_{t_0+}}x^0\right)\left(t+0\right),\left({}^c{D^{\alpha }_{t_0+}}x^0\right)\left(t-0\right)\ \right)=\]
\[={\mathcal Z}\left(t,\ \left({}^c{D^{\alpha }_{t_0+}}x^0\right)\left(t-0\right)\right)
-{\mathcal Z}\left(t,\ \left({}^c{D^{\alpha
}_{t_0+}}x^0\right)\left(t+0\right)\right)\ge 0,\ \ t\in A.\]
Using equality (38) from these inequalities, we obtain equality
(\ref{equ37}).
\end{proof}

\section{Conclusion}

In this paper, we study the problems of minimization of a
functional depending on the fractional Caputo derivative of order
$0<\alpha \leq 1$ and the fractional Riemann-Liouville integral of
order $\beta >0$ under fixed endpoints. Taking into account the
relationships between the parameters $\alpha$ and $\beta$, a
fractional analogue of the most important lemmas of the calculus
of variations, the fundamental Du Bois-Reymond lemma, is proved.
Using this lemma, the necessary Euler-Lagrange conditions in
integral form are obtained for the simplest fractional variational
problem and for the fractional isoperimetric problem. Using a
special variation depending on some numerical parameters, the
necessary first-order conditions for strong and weak local minima
are proved in the case of nonsmoth Lagrangians with respect to the
argument $^{c}D_{t_{0}+}^{\alpha}x$. From these necessary
conditions, as a consequence, we obtain the Weierstrass condition
and its local modification.
Unlike some works found in the literature, we prove the Legendre
conditions by the standard classical method via the Weierstrass
condition. In addition, the necessary Weierstrass-Erdmann
conditions at the corner points are obtained. Examples are
provided to illustrate the significance of the main results
obtained.

In this paper, for a simple fractional variational problem, we have proved three of the four fundamental conditions of the classical calculus of variations, namely the Euler-Lagrange, Legendre and Weierstrass conditions. It should be noted that the question of the applicability of the Jacobi condition to simple fractional variation problems still remains open.
Note that the method presented here can be used to derive
optimality conditions for other fractional calculus problems.
\\
\\
\textbf{Data availibility} No datasets were generated or analysed during the current study.
\subsection*{Declarations}
\textbf{Conflict of interest} The authors declare no potential Conflict of interest.


\bigskip

\begin{thebibliography}{99}

\bibitem{13}  
\newblock Gelfand IM,  Fomin SV.
\newblock \emph{Calculus of variations}. Revised English edition translated and
edited by Richard A. Silverman. Prentice-Hall, Inc., N.J.: Englewood
Cliffs, 1963.

\bibitem{17}  
\newblock Hestenes MR.
\newblock \emph{Calculus of variations and optimal control theory}.
\newblock Robert E. N.Y.: Krieger Publishing
Co., Inc., Huntington, 1980. Corrected reprint of the 1966
original.

\bibitem{27} 
\newblock Brunt B.
\newblock \emph{The calculus of variations}.
\newblock Universitext. New York: Springer-Verlag, 2004.

\bibitem{18} 
\newblock Hilfer R.
\newblock \emph{Applications of fractional calculus in physics}.
\newblock World Scientific, New
Jersey: River Edge, 2000.

\bibitem{19}  
\newblock Kilbas AA,  Srivastava HM,  Trujillo JJ.
\newblock \emph{Theory and applications of fractional differential
equations, volume 204 of North-Holland Mathematics Studies}.
\newblock Amsterdam: Elsevier Science B.V., 2006.

\bibitem{25} 
\newblock Sabatier J, Agrawal OP, Machado JAM.
\newblock  \emph{Advances in fractional calculus}.
\newblock Dordrecht: Springer, 2007.

\bibitem{26}  
\newblock Samko SG,  Kilbas AA,  Marichev OI.
\newblock \emph{Fractional integrals and derivatives}.
\newblock Yverdon: Gordon
and Breach Science Publishers,  1993.
\newblock Theory and
applications, edited and with a foreword by S. M. Nikolskiui,
translated from the 1987 Russian original, revised by the authors.

\bibitem{3} 
\newblock Almeida R,  Malinowska AB,  Torres DFM.
\newblock \emph{A fractional calculus of variations for
multiple integrals with application to vibrating string}.
\newblock J. Math. Phys. \textbf{51}(3):033503, 12 (2010)

\bibitem{21} 
\newblock Mahmudov EN, Yusubov ShSh.
\newblock \emph{Nonlocal boundary
value problems for hyperbolic equations with a Caputo fractional
derivative}.
\newblock J. Comput. Appl. Math. 2021; \textbf{398}:1-15.

\bibitem{28} 
\newblock Yusubov ShSh.
\newblock \emph{Boundary value problems for
hyperbolic equations with a Caputo fractional derivative}.
\newblock Adv.Math. Models Appl. 2020; \textbf{5}:192-204.


\bibitem{35}  
\newblock Zaslavsky G.M.
\newblock \emph{Hamiltonian chaos and fractional dynamics}.
Oxford University Press, Oxford,
2008. Reprint of the 2005 original.

\bibitem{24} 
\newblock Riewe F.
\newblock \emph{Nonconservative Lagrangian and Hamiltonian mechanics}.
\newblock Phys. Rev. E (3), 1996; \textbf{53}(2):1890-1899.

\bibitem{1}  
\newblock Agrawal OP.
\newblock \emph{Formulation of Euler-Lagrange equations for fractional variational problems}.
\newblock J. Math. Anal. Appl. 2002; \textbf{272}:368-379.

\bibitem{2}  
\newblock Almeida R,  Malinowska AB,  Torres D.
\newblock \emph{Fractional Euler-Lagrange differential equation via Caputo derivatives}.
\newblock Fract.Dynam.Contr. 2012; \textbf{II}:109-118.

\bibitem {6} 
\newblock Baleanu D, Trujillo J. 
\newblock \emph{On exact solutions of a class of
fractional Euler-Lagrange equations}.
\newblock Nonlin.Dynam. 2008; \textbf{52}:331-335.

\bibitem{8} 
\newblock Bourdin L, Ferreira RAC.
\newblock \emph{Legendre's necessary condition for fractional
Bolza functionals with mixed initial/final constraints}.
\newblock J.Opim. Theory Appl. 2021; \textbf{190}:672-708.

\bibitem{12} 
\newblock Ferreira RAC.
\newblock \emph{Fractional calculus of variations: a novel way to look at it}.
\newblock Fract. Calc. Appl.Anal. 2019; \textbf{22}: 1133-1144.

\bibitem{27Q}
\newblock Aydin M, Mahmudov NI.
\newblock \emph{A Study on Linear Prabhakar Fractional Systems with Variable Coefficients}.
\newblock Qual. Theory Dynam. Syst. 2024; \textbf{23}:210.


\bibitem{1M}	
\newblock Abbas MI, Ragusa MA. 
\newblock \emph{Nonlinear fractional differential inclusions with non-singular Mittag-Leffler kernel}. 
\newblock AIMS Mathematics 2022; \textbf{7}(11):20328-20340/


\bibitem{3M}
\newblock Agarwal RP, Alsaedi A, Alsharif A, Ahmad B. 
\newblock \emph{On nonlinear fractional-order boundary value problems with nonlocal multipoint conditions involving Liouville-Caputo derivatives}.
\newblock Differ Equ Appl 2017; \textbf{9}(2):147-160

\bibitem{27M}	
\newblock Rizwan R,  Zada A,  Wang X.
\newblock \emph{Stability analysis of nonlinear implicit fractional Langevin equation with noninstantaneous impulses}.
\newblock Adv. Differ. Equ.  2019; \textbf{2019}:85.


\bibitem{5M}	
\newblock Baleanu D, Fernandez A. 
\newblock \emph{On some new properties of fractional derivatives with Mittag-Leffler kernel}.
\newblock Commun. Nonlin. Sci. Numer. Simul. 2018; \textbf{59}:444-462.

\bibitem{41M}	
\newblock Zeng S, Bouach A, Haddad T. 
\newblock \emph{On Nonconvex Perturbed Fractional Sweeping Processes}.
\newblock Appl. Math. Optim. 2024; \textbf{89}(3):1-36.

\bibitem{30M}	
\newblock Taneja K, Deswal K, Kumar D, J Vigo-Aguiar. 
\newblock \emph{A Robust and higher order numerical technique for a time-fractional equation with nonlocal kondüitin}.
\newblock J. Math. Chemistry 2025; \textbf{63}:626–649.


\bibitem{20M}	
\newblock Jleli M, Kirane M, Samet B. 
\newblock \emph{A derivative concept with respect to an arbitrary kernel and applications to fractional calculus}.
\newblock Math. Meth. Appl. Sci. 2019; 42(1):137-160.

\bibitem{31M}
\newblock Vellappandi M, Venkatesan Govindaraj, José Vanterler da C. Sousa. 
\newblock \emph{Fractional optimal reachability problems with $\phi$-Hilfer fractional derivative}. \newblock Math. Meth. Appl. Sci. 2022; \textbf{45}(10):6255-6267.


\bibitem{14}  
\newblock Guo TL.
\newblock \emph{The necessary conditions of fractional optimal control in the sense of Caputo}.
\newblock J.Optim. Theory Appl. 2013; \textbf{156}:115-126.


\bibitem{29}  
\newblock Yusubov ShSh,  Mahmudov EN.
\newblock \emph{Optimality conditions of singular controls for systems with
Caputo fractional derivatives}.
\newblock J. Indust. Manag. Optim. 2023; \textbf{19}(1):246-264.

\bibitem{30}  
\newblock Yusubov ShSh,  Mahmudov EN.
\newblock \emph{Necessary
optimality conditions for quasi-singular controls for systems with
 Caputo fractional derivatives}.
 \newblock Archiv. Contr.Sci. 2023; \textbf{33}(3):463-496.
 
\bibitem{31}  
\newblock Yusubov ShSh,  Mahmudov EN. 
\newblock \emph{Necessary and
sufficient optimality conditions for fractional
Fornasini-Marchesini model}.
\newblock J. Indust. Manag. Optim. 2023; \textbf{19}(10):7221-72444. 

\bibitem{32} 
\newblock Yusubov ShSh,  Mahmudov EN.
\newblock \emph{Some necessary
optimality conditions for systems with fractional Caputo
derivatives}.
\newblock J. Indust. Manag. Optim. 2023; \textbf{19}(12):8831-8850.

\bibitem{33} 
\newblock Yusubov ShSh,  Mahmudov EN.
\newblock \emph{Pontryagin's
maximum principle for the Roesser model with a fractional Caputo
derivative}.
\newblock Arch. Contr. Sci. 2024; \textbf{34}(2):271-300.

\bibitem{34} 
\newblock Yusubov ShSh,  Mahmudov EN.
\newblock \emph{Necessary
optimality conditions for singular controls of Caputo fractional
systems with delay in control}.
\newblock Qual. Theory Dynam. Syst. 2025; \textbf{24}(2):1-32. 

\bibitem{15M}	
\newblock  Harrat A, Nieto JJ, Debbouche A. 
\newblock \emph{Solvability and optimal controls of impulsive Hilfer fractional delay evolution inclusions with Clarke subdifferential}.
\newblock J. Comput. Appl. Math. 2018; \textbf{344}:725-737.

\bibitem{20MM}
\newblock Kamocki R. 
\newblock \emph{Pontryagin maximum principle for fractional ordinary optimal control problems}.
\newblock Math. Meth. Appl. Sci. 2014;\textbf{37}(11):1668-1686.

\bibitem{36E}
\newblock Mahmudov EN.
\newblock \emph{Approximation and Optimization of Discrete and Differential Inclusions}.
\newblock Elsevier: Boston, USA, 2011.

\bibitem{7} 
\newblock Bergounioux M, Bourdin L. 
\newblock \emph{Pontryagin maximum principle for general Caputo fractional
optimal control problems with Bolza cost and terminal constraints}.
\newblock ESAIM Contr. Optim. Calc. Var. 2020; \textbf{26}:38.

\bibitem{11} 
\newblock Ekeland I.
\newblock \emph{On the variational principle}.
\newblock J. Math. Anal. Appl. 1974; \textbf{47}:324-353.

\bibitem{9} 
\newblock Bourdin L.
\newblock \emph{Contributions to optimal control theory
with fractional and time scale calculi, and to variational
analysis in view of shape optimization problems in contact
mechanics}.  1-183, 2020.








































\end{thebibliography}

\end{document}